\documentclass[12pt,a4paper]{article}

\usepackage[utf8]{inputenc}
\usepackage[english]{babel}
\usepackage[a4paper,margin=1in]{geometry}

\usepackage{sectsty}
\sectionfont{\normalsize}
\subsectionfont{\small}
\subsubsectionfont{\small}
\usepackage{amssymb}
\usepackage{amsfonts}
\usepackage{amsmath}
\usepackage{amsthm}

\usepackage{epsfig}
\usepackage{graphicx}
\usepackage{pgfplots}

\usepackage{dsfont}
\usepackage{mathrsfs}
\usepackage{latexsym}

\usepackage{verbatim}
\usepackage{setspace}
\usepackage{enumitem}
\usepackage{mathtools}

\numberwithin{equation}{section}

\renewcommand{\labelenumi}{\textup{(\roman{enumi})}}
\renewcommand\theenumi\labelenumi 

\theoremstyle{plain}
\newtheorem{theorem}{Theorem}[section]
\newtheorem{corollary}[theorem]{Corollary}
\newtheorem{lemma}[theorem]{Lemma}

\theoremstyle{definition}

\newtheorem{remark}[theorem]{Remark}

\newtheorem{example}[theorem]{Example}

\newcommand\diff{\mathrm{d}}
\newcommand{\real}{\mathds{R}}
\newcommand{\nat}{\mathds{N}}

\newcommand{\integer}{\mathds{Z}}

\newcommand{\PP}{\mathds{P}}

\newcommand{\EE}{\mathds{E}}

\newcommand{\cC}{\mathcal{C}}

\newcommand{\cF}{\mathcal{F}}

\newcommand{\cL}{\mathcal{L}}

\newcommand{\ee}{\mathrm{e}}
\newcommand{\bone}{\mathds{1}}

\newcommand{\BF}{\operatorname{BF}}
\newcommand{\CM}{\operatorname{CM}}

\newcommand{\id}{\operatorname{id}}
\newcommand{\supp}{\operatorname{supp}}

\newcommand{\loc}{\mathrm{loc}}

\renewcommand{\epsilon}{\varepsilon} 

\RequirePackage[colorlinks=false,citecolor=blue,urlcolor=blue]{hyperref}

\usepackage[final]{showkeys}
\usepackage{xcolor}
\usepackage[normalem]{ulem} 

\definecolor{see}{HTML}{0000AA}
\definecolor{doubt}{HTML}{AF0000}
\definecolor{todo}{HTML}{0000AA}

\usepackage{marginnote}
\begin{document}
\frenchspacing
\allowdisplaybreaks[4]

\title{\bfseries On $q$-scale functions of spectrally negative L\'evy processes}
\author{Anita Behme\thanks{Technische Universit\"at
		Dresden, Institut f\"ur Mathematische Stochastik, Helmholtzstra{\ss}e 10, 01069 Dresden, Germany. \texttt{anita.behme@tu-dresden.de},  \texttt{david.oechsler@tu-dresden.de}, \texttt{rene.schilling@tu-dresden.de}},\;  David Oechsler$^\ast$ and Ren\'{e} L.\ Schilling$^\ast$}
\date{\small To appear in \emph{Advances in Applied Probability} 55.1 (March 2023)}
\maketitle

\begin{abstract}
We obtain series expansions of the $q$-scale functions of arbitrary spectrally negative L\'{e}vy processes, including processes with infinite jump activity, and use these to derive various new examples of explicit $q$-scale functions. Moreover, we study smoothness properties of the $q$-scale functions of spectrally negative L\'{e}vy processes with infinite jump activity. This complements previous results of Chan et al. \cite{ChanKypSavov} for spectrally negative L\'{e}vy processes with Gaussian component or bounded variation.


\end{abstract}

\section{Introduction}\label{S0}


Spectrally negative L\'evy processes, i.e.\ L\'evy processes whose jumps are negative, play a special role in applied mathematics since many real-life applications in queueing, risk, finance, etc.,\ show only jumps or shocks in one direction. Moreover, the restriction to one-sided jumps also has mathematical advantages since spectrally negative L\'{e}vy processes admit $q$-scale functions.
The $q$-scale functions get their name from the related scale functions of regular diffusions and, just as for diffusions, many fluctuation identities of spectrally one-sided L\'{e}vy processes may be expressed in terms of $q$-scale functions, see e.g.\ \cite{berto97, emery73, roger90, roger00} or \cite{doney07, kuzne13}, and the references given there.

Let $(L_t)_{t\geq 0}$ be a spectrally negative L\'evy process with Laplace exponent $\psi$. For every $q\geq 0$ the $q$-scale function is defined as the unique function $W^{(q)}:\real \to\real_+$ such that $W^{(q)}(x)=0$ for all $x<0$ and whose Laplace transform is given by
\begin{align}\label{def-scalefunction}
	\cL[W^{(q)}](\beta) = \frac1{\psi(\beta)-q},\quad \beta>\Phi(q) :=\sup\{y: \psi(y)=q\}.
\end{align}
However, it is not often possible to invert the Laplace transform explicitly and to obtain $W^{(q)}$ in a closed-form expression; see e.g.\ \cite{hubal10} for a collection of known cases. Although there are  efficient methods to evaluate $q$-scale functions numerically, cf.\ \cite{kuzne13}, or to approximate them, cf.\ \cite{egami14}, it is of considerable interest to expand the set of examples with analytic expressions.  Quite often, evaluating the scale function is only an intermediate step, in the sense that one needs functions of $W^{(q)}$, rather than $W^{(q)}$ itself. This shows the limitations of numerical or approximative approaches, e.g.\ if one has to estimate the jump measure, which is the case in most applications.

Recently, Landriault \& Willmot \cite{landr20} proposed the following method to derive an explicit representation of the $q$-scale function of a spectrally negative L\'{e}vy process \emph{with finite L\'evy measure}, i.e.\ of a spectrally negative compound Poisson process: determine suitable functions $f,g:(\Phi(q),\infty)\to\real$ and a constant $b\neq0$ to rewrite the right-hand side of \eqref{def-scalefunction} in a form which can be expanded into a geometric series, i.e.\ such that
\begin{align}\label{eq-geo_series}
	\frac{\cL g(\beta)}{b-\cL f(\beta)}
    = \cL g(\beta)\sum_{n=0}^\infty \frac{(\cL f(\beta))^n}{b^{n+1}}
\end{align}
on the set $\{\beta\geq 0: |\cL f(\beta)|<b\}$. This Laplace transform can then easily be inverted term-by-term. The same idea was already used in \cite{chan11} in order to derive smoothness properties of the $q$-scale functions of spectrally negative L\'{e}vy processes with paths of bounded variation or with Gaussian component.

In Section~\ref{S4} of the present article we show how one can modify the approach of \cite{landr20} to obtain series representations of the $q$-scale functions of spectrally negative L\'evy processes with infinite jump activity. Moreover, at least for certain subclasses, that have not been treated in \cite{chan11}, we study the smoothness of scale functions in Section~\ref{S5}.  The final Section~\ref{S3} contains several auxiliary technical lemmas which we need for our proofs.


\section{Preliminaries}\label{S2}


    Most of our notation is standard or self-explanatory. Throughout, $\nat = \{0,1,2,.\dots\}$ are the natural numbers starting from $0$, and $\real_+=[0,\infty)$ is the non-negative half-line.  We write $\partial_x:=\frac{\partial}{\partial x}$, and we drop the subscript $x$ if no confusion is possible. By $\bone$, and $\id$, we mean the constant function $x\mapsto 1$, and the identity map $x\mapsto x$ on $\real_+$, respectively. If needed, we extend a function $f:(0,\infty)\to\real$ onto $[0,\infty)$ by setting $f(0):= f(0+) = \lim_{  y\downarrow 0}f(y)$ if the limit exists in $\real$, and $f(0):=0$ otherwise. In particular, this ensures that $f\in L^1_\loc(\real_+)$ implies integrability at $0$.

\subsection{Laplace transforms and convolutions}

Throughout, $\cL$ denotes the (one-sided) Laplace transform
\begin{align}
	\cL\mu(\beta) := \cL[\mu](\beta) := \int_{[0,\infty)} \ee^{-\beta x} \,\mu(\diff x),
\end{align}
for a measure $\mu$ on $[0,\infty)$ and every $\beta\geq 0$ for which the integral converges. As usual, if $\mu$ is absolutely continuous w.r.t.\ Lebesgue measure with a locally integrable density $f$, i.e.\ $\mu(\diff x) = f(x)\,\diff x$, we write $\mu\ll \diff x$, identify $\mu$ and $f$, and write $\cL f=\cL\mu$.

Let $\mu$ and $\nu$ be two measures with support in $\real_+$ such that the Laplace transforms $\cL\mu$ and $\cL\nu$ converge for some $\beta_0>0$. Then they also converge for all $\beta\in[\beta_0,\infty)$ and we have
\begin{align}\label{eq-lapl_convo}
	\cL[\mu*\nu](\beta) = \cL\mu(\beta) \cL\nu(\beta) \quad\text{on\ \ }[\beta_0,\infty),
\end{align}
for the convolution of the measures
\begin{align}
	(\mu*\nu)(B):=\int_{\real}\int_{\real} \bone_B(x+y) \,\mu(\diff x)\,\nu(\diff y),\quad\text{for all Borel sets\ \ } B\subset\real.
\end{align}
If $\mu$ and $\nu$ are absolutely continuous w.r.t.\ Lebesgue measure with locally integrable densities $f$ and $g$, then \eqref{eq-lapl_convo} becomes
\begin{align*}
	\cL[f*g](\beta) = \cL f(\beta) \cL g(\beta),
\end{align*}
where the convolution is given by
\begin{align*}
	f*g(x)
	= \int_\real f(x-y)g(y)\,\diff y
	= \int_{[0,x]} f(x-y)g(y)\,\diff y,
	\quad x\geq 0.
\end{align*}
The second equality is due to the fact that the supports of $f$ and $g$ are contained in $\real_+$. If $f,g\in L^1_{\loc}(\real_+)$, then for all $x\geq 0$
\begin{align}\label{eq-Young}
	\|f*g\|_{L^1([0,x])}\leq \|f\|_{L^1([0,x])}\|g\|_{L^1([0,x])},
\end{align}
where $\|f\|_{L^1([0,x])}:=\int_{[0,x]} |f(x)|\,\diff x$. Finally, if $f\in AC[0,\infty)$ or $f\in AC(0,\infty)$ is absolutely continuous, integration by parts yields
\begin{align}\label{eq-LapDif}
	\cL[\partial f](\beta)=\beta \cL f(\beta)-f(0+).
\end{align}

\subsection{Spectrally negative L\'evy processes }
	
Throughout this article, $(L_t)_{t\geq 0}$ denotes a spectrally negative L\'{e}vy process,  \((\mathbb{P}_x)_{x\in\real}\) the law of \((L_t)_{t\geq0}\), and \((\cF_t)_{t\geq0}\) the natural filtration satisfying the usual assumptions. A spectrally negative L\'{e}vy process is   a one-dimensional L\'{e}vy process with no positive jumps and Laplace exponent given by
\begin{align*}
    \psi(\beta)
    := \frac{1}{t} \log \mathbb{E}\left[\ee^{\beta L_t}\right]
    = c\beta + \sigma^2\beta^2 + \int_{(-\infty,0)}\left(\ee^{\beta x}-1-\beta x  \bone_{(0,1)}(|x|) \right) \widetilde\nu(\diff x),
    \quad \beta\geq 0,
\end{align*}
where $(c, 2 \sigma^2, \widetilde{\nu})$ is the characteristic triplet of $(L_t)_{t\geq 0}$. To simplify notation, we denote by $\nu$ the reflection of $\widetilde\nu$, i.e.\ $\nu(B)=\widetilde\nu(-B)$ for all measurable sets $B\subset\real$, which yields
\begin{align}\label{eq-LaplaceExponent}
    \psi(\beta)
    = c\beta + \sigma^2\beta^2 + \int_{(0,\infty)} \left(\ee^{-\beta x}-1+\beta x\bone_{(0,1)}(x) \right) \nu(\diff x),
    \quad \beta\geq 0.
\end{align}
Depending on the characteristics $(c, 2 \sigma^2, \widetilde{\nu})$ we may simplify the representation of the Laplace exponent $\psi$, see Lemma~\ref{lem-ExpRep}. Although such representations frequently appear in the literature, we provide the short arguments, also to fix notation and ideas.
\begin{lemma}\label{lem-ExpRep}
    Let $(L_t)_{t\geq 0}$ be a spectrally negative L\'{e}vy process with characteristic triplet $(c, 2\sigma^2,\widetilde\nu)$ and Laplace exponent $\psi$.
	\begin{enumerate}
	\item\label{lem-ExpRep-i} \textup{[}jumps of bounded variation\textup{]}
		If $\int_{(0,1)} y\,\nu(\diff y) < \infty$, then
		\begin{gather*}
		  \psi(\beta) = c'\beta + \sigma^2\beta^2 - \beta\cL\bar\nu(\beta),
		\end{gather*}
		with $c'=c+\int_{(0,1)} y\,\nu(\diff y)$ and $\bar\nu(x):= \nu([x,\infty))$ being the tail function of $\nu$.

	\item\label{lem-ExpRep-ii} \textup{[}finite first moment\textup{]}
		If $\int_{[1,\infty)} y\,\nu(\diff y) < \infty$, then
		\begin{gather*}
			\psi(\beta) = c''\beta + \sigma^2\beta^2 + \beta^2\cL\bar{\bar\nu}(\beta),
		\end{gather*}
        with $c'' = c - \int_{[1,\infty)} y \,\nu(\diff y)  = \mathbb{E} L_1$ and $\bar{\bar\nu}(x) := \int_{ [x,\infty)} \bar\nu(y)\,\diff y$ being an integrated tail function of $\nu$.
	\end{enumerate}	
\end{lemma}
\begin{proof}
To see~\ref{lem-ExpRep-i} set $c'=c+\int_{(0,1)} y\,\nu(\diff y)$, and compute
\begin{align*}
    \psi(\beta)
    &=c'\beta+\sigma^2\beta^2 +\int_{(0,\infty)} (\ee^{-\beta x}-1)\nu(\diff x) \\
    &= c'\beta+\sigma^2\beta^2 -\beta\int_{ (0,\infty)} \int_{ [0,x]} \ee^{-\beta y}\,\diff y~ \nu(\diff x)\\
	&=c'\beta+\sigma^2\beta^2 -\beta\int_{ [0,\infty)}\ee^{-\beta y}\int_{[y,\infty)}\nu(\diff x) ~\diff y \\
    &=c'\beta+\sigma^2\beta^2 -\beta\cL\bar\nu(\beta).
\end{align*}
Further, \ref{lem-ExpRep-ii} follows in a similar way if we set $c''=c-\int_{[1,\infty)} y\,\nu(\diff y)$ and use
\begin{align*}
    \psi(\beta)
    &=c'' \beta + \sigma^2 \beta^2 +\int_{ (0,\infty)}\left( \ee^{-\beta x} -1+\beta x\right) \nu(\diff x)\\
	&=c''\beta+\sigma^2\beta^2 -\beta\int_{ (0,\infty)} \int_{[0,x]}( \ee^{-\beta y}-1)\,\diff y~ \nu(\diff x)\\
	&=c''\beta+\sigma^2\beta^2 +\beta^2\int_{ (0,\infty)} \int_{[0,x]}\int_{[0,y]} \ee^{-\beta z} \,\diff z~\diff y~ \nu(\diff x)\\
	&=c''\beta+\sigma^2\beta^2 -\beta^2\int_{ [0,\infty)} \ee^{-\beta z} \int_{ [z,\infty)}\int_{[y,\infty)} \nu(\diff x)~\diff y~\diff z\\
	&=c''\beta+\sigma^2\beta^2+\beta^2\cL\bar{\bar\nu}(\beta).
\end{align*}
This completes the proof.
\end{proof}	

From now on we will always use the constant $c\in\real$ to denote the location parameter if $\psi$ is in the form \eqref{eq-LaplaceExponent}, and $c',c''\in\real$ if we write $\psi$ as in Lemma~\ref{lem-ExpRep} \ref{lem-ExpRep-i} and~\ref{lem-ExpRep-ii}.

Note that, whenever $\int_{[1,\infty)}y\nu(\diff y)<\infty$, the function $\bar{\bar\nu}$ is well-defined, and integration by parts implies that in this case
\begin{gather*}
    \int_{(0,1)}y\nu(\diff y)=\infty \quad
    \Longleftrightarrow \quad \bar{\bar\nu}(0+)=\infty.
\end{gather*}

Further background information on (spectrally negative) L\'evy processes can be found in \cite{berto98, sato13}.

\subsection{Scale functions}\label{sub-se-23}

The $q$-scale function of a spectrally negative L\'evy process $(L_t)_{t\geq 0}$ is uniquely determined via its Laplace transform given in \eqref{def-scalefunction}.
If $\widetilde\nu \equiv 0$, i.e.\ if $(L_t)_{t\geq 0}$ is a Brownian motion with drift, the Laplace transform in \eqref{def-scalefunction} can be inverted explicitly, cf.\ \cite{kuzne13}, leading to the well-known formula
\begin{gather}\label{eq-scaleBM}
	W^{(q)}(x)
	= \frac{1}{\sqrt{c^2 +2q\sigma^2}} \left(\ee^{\left(\sqrt{c^2 +2q\sigma^2} - c\right){x}/{\sigma^2}} - \ee^{-\left(\sqrt{c^2 +2q\sigma^2} + c\right){x}/{\sigma^2}} \right), \quad x\geq 0.
\end{gather}

We will, therefore, restrict ourselves to processes with $\widetilde\nu \not\equiv 0$. Moreover, we will exclude subordinators, i.e.\  monotone L\'evy processes. In particular, if $\sigma^2=0$ and $\int_{(0,1)} x\,\nu(\diff x)<\infty$, we assume  $c\geq \int_{(0,1)} x\,\nu(\diff x)$.

  For a spectrally negative L\'{e}vy process \(L=(L_t)_{t\geq0}\) with Laplace exponent \(\psi\) and characteristic triplet \((c,2\sigma^2,\widetilde\nu)\) it is well-known (cf. \cite[Section~3.3]{kuzne13}, \cite{chan11}, \cite[Chapter 8.2]{kypri14}) that under the exponential change of measure
\begin{align*}
\frac{\diff \PP_x^{\Phi(q)}}{\diff \PP_x}\Bigg|_{\cF_t}=\ee^{\Phi(q)(L_t-x)-qt}, \quad q\geq0,
\end{align*}
the process \((L,\PP^{\Phi(q)})\) is again a spectrally negative L\'{e}vy process whose Laplace exponent is given by \(\psi_{\Phi(q)}(\beta)=\psi(\beta+\Phi(q))-q\) and whose L\'{e}vy measure is \(\widetilde\nu_{\Phi(q)}(\diff x)=\ee^{\Phi(q)x}\widetilde\nu(\diff x)\). The shape of \(\psi_{\Phi(q)}\), together with \eqref{def-scalefunction} yields \(W^{(q)}(x)=\ee^{\Phi(q)x}W_{\Phi(q)}(x)\), where \(W_{\Phi(q)}\) denotes the \(0\)-scale function of \((L,\PP^{\Phi(q)})\). This relation often allows us to assume without loss of generality that \(q=0\) and we will use this observation below in Section \ref{sub-se-41} for the proof of Theorem \ref{thm-smoothness}.\\
  The surveys \cite{avram20} and \cite{kuzne13} contain a thorough introduction to scale functions and some of their applications.

\section{Series expansions of $q$-scale functions} \label{S4}


As before, $(L_t)_{t\geq 0}$ is a spectrally negative L\'evy process with Laplace exponent $\psi$ and L\'evy triplet $(c, 2\sigma^2,\widetilde\nu)$;  $\nu(B):=\widetilde\nu(-B)$, we set $\bar\nu(x) = \nu([x,\infty))$ and $\bar{\bar\nu}(x) = \int_{[x,\infty)} \bar\nu(y)\,\diff y$; $c'$ and $c''$ are the constants from Lemma~\ref{lem-ExpRep}.

We will extend the approach of \cite{landr20} to derive series expansions for the $q$-scale function of $(L_t)_{t\geq 0}$, that will be presented in Theorems \ref{thm-brownian}, \ref{thm-bounded}, and \ref{thm-unbdd} below. These three theorems cover the following three different situations:

\begin{enumerate}
\item $\sigma^2>0$ (Section~\ref{S4a}, Theorem~\ref{thm-brownian}): In this case, our result extends \cite[Thm.~2.4]{landr20} from finite to arbitrary jump measures.

\item $\sigma^2=0$ and $\int_{(0,1)} y\,\nu(\diff y)<\infty$ (Section~\ref{S4b}, Theorem~\ref{thm-bounded}): We recover results which have been earlier obtained  in \cite{doeri11} and, in special cases, also in \cite{chan11,erlan09,landr20}.

\item $\sigma^2=0$ and $\int_{(0,1)} y\,\nu(\diff y) = \infty$ (Section~\ref{S4c}, Theorem~\ref{thm-unbdd}): Our results in this situation seem to be new.
\end{enumerate}

The proofs of all three theorems use the same strategy: find a suitable expansion of the Laplace transform, which can be inverted term-by-term.

\subsection{L\'evy processes with a Gaussian component}\label{S4a}

We start by presenting a series expansion for scale functions of spectrally negative L\'evy processes that have a Gaussian component.   Results of this kind have a long history. For example, if we combine in \cite{chan11} the proof of Theorem 1 with Corollary 9, we arrive at a series representation for $\partial W^{(0)}$, and this leads to an (implicit) expansion for $\partial W^{(q)}$ using the relations between $W^{(0)}$ and $W^{(q)}$. Also note that, setting $\nu\equiv 0$ in Theorem \ref{thm-brownian} below, we recover \eqref{eq-scaleBM}.   At this point we note that the moment condition $\int_{[1,\infty)} y\,\nu(\diff y)<\infty$ appearing in Theorem \ref{thm-brownian} is for notational convenience only, and can be discarded at the expense of heavier notation, see also Remark \ref{rem-addcond} below.

\begin {theorem}\label{thm-brownian}
    Let $(L_t)_{t\geq 0}$ be a spectrally negative L\'{e}vy process with triplet $(c, 2 \sigma^2,\widetilde\nu)$ such that $\sigma^2>0$ and $\int_{[1,\infty)} y\,\nu(\diff y)<\infty$. For $q\geq 0$ set
    \begin{gather*}f(x) :=  -c''+ qx - \bar{\bar\nu}(x) = \textstyle\int_{[1,\infty)} y\,\nu(\diff y) - c + qx - \bar{\bar\nu}(x), \quad x\geq0,
    \end{gather*}
    then the $q$-scale function of $(L_t)_{t\geq 0}$ is  given by
	\begin{align}\label{case-brownian}
		W^{(q)} = \id *\sum_{n=0}^\infty \frac{f^{*n}}{(\sigma^2)^{n+1}},
	\end{align}
	and the series on the right-hand side converges uniformly on compact subsets of $\real_+$ to a limit in $L^1_{\loc}(\real_+)$.
\end {theorem}
\begin{proof}
    By Lemma~\ref{lem-ExpRep} \ref{lem-ExpRep-ii} the Laplace exponent of $L$ is given by $\psi(\beta)=c''\beta+\sigma^2\beta^2+\beta^2\cL\bar{\bar\nu}(\beta)$ with $c''=c-\int_{[1,\infty)} y\,\nu(\diff y)$. From the definition of the scale function we thus get for all $q\geq 0$
	\begin{equation}\label{eq-proof1}\begin{aligned}
			\cL[W^{(q)}](\beta)
			&=\frac1{\beta^2\left(c''\beta^{-1} + \sigma^2 + \cL\bar{\bar\nu}(\beta)-q\beta^{-2}\right)}. \nonumber
	\end{aligned}\end{equation}
    For $\beta>1$ large enough we have $\sigma^2 > |c''\beta^{-1} + \cL\bar{\bar\nu}(\beta) -q\beta^{-2}|$, and so we may expand the fraction appearing on the right hand side into a geometric series:
	\begin{align}\label{eq-proof2}
        \cL[W^{(q)}](\beta)
        &=\frac1{\beta^2} \sum_{n=0}^\infty \frac{\left(-c''\beta^{-1} - \cL\bar{\bar\nu}(\beta) + q\beta^{-2}\right)^n}{(\sigma^2)^{n+1}}
        =\frac1{\beta^2}\sum_{n=0}^\infty \frac{\left(\cL f(\beta)\right)^n}{(\sigma^2)^{n+1}}
	\end{align}
	with $f(x) = -c'' - \bar{\bar\nu}(x)+ qx $ as in Theorem~\ref{thm-brownian}, and for sufficiently large $\beta > 1$.
	
    Note that it is enough to know the (existence of the) Laplace transform for large values of $\beta\in[\Phi(q),\infty)$ in order to characterize $W^{(q)}$, see e.g.\ \cite[Thm.~6.3]{widde15}. Thus, to prove the representation for $W^{(q)}$, it remains to show that the summation and the Laplace transform in \eqref{eq-proof2} can be interchanged. To do so, let us write for a moment $g(x) := \sigma^{-2}f(x)$. Then for every $m\in\nat$
	\begin{align*}
		\left| \cL\left(\sum_{n=0}^\infty g^{*n}\right)-\sum_{n=0}^m \cL\left(g^{*n}\right)\right|
		&= \left| \cL\left(\sum_{n=0}^\infty g^{*n}\right)-\cL\left(\sum_{n=0}^m g^{*n}\right)  \right|\\
		&\leq \cL\left(\sum_{n=m+1}^\infty |g|^{*n}\right)
		= \sum_{n=m+1}^\infty \cL[|g|]^n.
	\end{align*}
    For sufficiently large arguments $\beta > 1$ we know that $\cL[|g|](\beta)<1$, and so the series on the right-hand side converges, which allows us to interchange the summation and the Laplace transform. Moreover, the above calculation shows that $\sum_{n=1}^\infty g^{*n}$ converges in the weighted space $L^1 (\ee^{-\beta y}\,\diff y)$, and so it is well-defined in, say, $L^1_{\loc}(\real_+)$.
    Further, from H\"{o}lder's inequality it follows that
	\begin{align*}
		\left|\left (\id * \sum_{n=1}^\infty g^{*n}\right)(x)\right| \leq
		x\left\|\sum_{n=1}^\infty g^{*n}\right\|_{L^1[0,x]}, \quad x\geq0.
	\end{align*}	
    By Lemma \ref{lem-ConPowDecay}, the series converges uniformly on compact subsets of $\real_+$. Finally, since $\cL[\id](\beta) = \beta^{-2}$ and $\id \ast g^{\ast 0} = \id \ast \delta_0=\id$, we get
	\begin{gather*}
	W^{(q)} = \id * \sum_{n=0}^\infty \frac{f^{*n}}{(\sigma^2)^{n+1}},
	\end{gather*}
	as stated.
\end{proof}

\begin{remark}\label{rem-addcond}
    The condition $\int_{[1,\infty)} y\,\nu(\diff y)<\infty$ ensures that the integrated tail $\bar{\bar\nu}$ is finite. We can obtain similar results without this condition, if we split the L\'{e}vy measure into $\nu^{(0)}=\nu\vert_{[0,z)}$ and $\nu^{(\infty)}=\nu\vert_{[z,\infty)}$ for some fixed $z\geq 1$, and replace $\bar{\bar\nu}$ by $\bar{\bar\nu}^{(0)}(x)=\int_{[x,z)} \bar\nu(y)\,\diff y$.
    The Laplace exponent $\psi$ is then given by
	\begin{gather*}
		\psi(\beta)
		= c''_{z}\beta + \sigma^2\beta^2 + \beta^2\cL[\bar{\bar\nu}^{(0)}](\beta) + \cL[\nu^{(\infty)}](\beta) - \|\nu^{(\infty)}\|,
	\end{gather*}
	where $c''_{z} := c - \int_{[1,\infty)} x\,\nu^{(0)}(\diff x)$ and $\|\nu^{(\infty)}\| = \nu^{(\infty)}\left([0,\infty)\right)$ is the total mass. Thus, with the same argument as in the proof of Theorem~\ref{thm-brownian},
	\begin{align}
		W^{(q)} = \id *\sum_{n=0}^\infty \frac{f^{*n}}{(\sigma^2)^{n+1}},
	\end{align}
    with $f$ such that
    \begin{gather*}
        \cL f(\beta)= -c_z'' \beta^{-1} - \cL[\bar{\bar\nu}^{(0)}](\beta) - \beta^{-2} \cL[\nu^{(\infty)}](\beta) + \beta^{-2}(q+\|\nu^{(\infty)}\|),
    \end{gather*}
    i.e.\ with $f$ given by
    \begin{gather*}
        f(x)= -c_z'' - \bar{\bar\nu}^{(0)}(x) - \int_{[z,x]} ( x-y )\nu^{(\infty)}(\diff y) + (q+\|\nu^{(\infty)}\|)x, \quad x\geq0.
    \end{gather*}
Note that the integral term in $f$ vanishes on $[0,z)$; therefore, it can be omitted when evaluating $W^{(q)}$ on $[0,z)$.
The above procedure thus yields a series expansion of the same form as \eqref{case-brownian} for  bounded subsets of $\real_+$, if we pick $z\geq 0$ sufficiently large and adjust $c''$ and $\bar{\bar\nu}^{(0)}$ accordingly.
\end{remark}

\begin{remark}\label{rem-reform}
     A different construction of the series expansions yields further useful formulae for the $q$-scale function. Consider the setting of Theorem~\ref{thm-brownian}. Instead of \eqref{eq-proof1} we can use
	\begin{align*}
		\cL[W^{(q)}](\beta)
		= \frac1{c''\beta + \sigma^2\beta^2 + \beta^2\cL\bar{\bar\nu}(\beta)-q}
		= \sum_{n=0}^\infty \frac{(-1)^n(\beta^2\cL\bar{\bar\nu}(\beta))^n}{(c''\beta + \sigma^2\beta^2 - q)^{n+1}}.
	\end{align*}
    Denote by $a_{1,2}=-\frac{c''}{2\sigma^2}\pm\sqrt{(\frac{c''}{2\sigma^2})^2+\frac{q}{\sigma^2}}$ the roots of the polynomial in the denominator. This allows us to invert the Laplace transform, leading to
	\begin{align*}
		W^{(q)}
		=\sum_{n=0}^\infty \frac{(-1)^n}{(\sigma^2)^{n+1}} \cL^{-1}\left(\frac{\beta^n}{(\beta+a_1)^{n+1}}\right) * \cL^{-1}\left(\frac{\beta^n}{(\beta+a_2)^{n+1}}\right) * \bar{\bar\nu}^{*n},
	\end{align*}
	with the inverse Laplace transforms of the rational functions given by
	\begin{align*}
		\cL^{-1}\left(\frac{\beta^n}{(\beta+a)^{n+1}}\right)=\frac1{n!}\partial^n\left(x^n\ee^{-ax}\right)
		\quad\text{for all\ \ } n\in\nat,\; a\in\real,
	\end{align*}
cf.\ \cite{widde15}. If, moreover, the L\'{e}vy measure $\nu$ has finite total mass $\|\nu\|<\infty$, we get
	\begin{gather*}
		\cL[W^{(q)}](\beta)
		= \frac1{c'\beta+\sigma^2\beta^2+\cL\nu(\beta)-(q+\|\nu\|)}
		= \sum_{n=0}^\infty \frac{(-1)^n(\cL\nu(\beta))^n}{(c'\beta+\sigma^2\beta^2-(q+\|\nu\|))^{n+1}}
	\end{gather*}
	with $c'=c+\int_{(0,1)} x\,\nu(\diff x)$. Thus, we have
	\begin{gather}\label{eq-landriault}
		W^{(q)}
		= \sum_{n=0}^\infty \frac{(-1)^n}{(\sigma^2)^{n+1}}\left(\frac{x^{n}}{n!}\ee^{-a_3 x}\right) *
		\left(\frac{x^{n}}{n!}\ee^{-a_4 x}\right)*\nu^{*n}
	\end{gather}
	with $a_{3,4}=-\frac{c'}{2\sigma^2}\pm\sqrt{(\frac{c'}{2\sigma^2})^2+\frac{q+\|\nu\|}{\sigma^2}}$. The convolution
	$
		\left(\frac{1}{n!}x^{n}\ee^{-a_3 x}\right)*\left(\frac{1}{n!}x^{n}\ee^{-a_4 x}\right)
	$	
can be explicitly worked out. Rather than stating the lengthy expression, we refer to \cite{landr20} where the case of finite L\'{e}vy measures is thoroughly studied for both $\sigma^2>0$ and $\sigma^2=0$.
\end{remark}

We close this section with two explicit examples in the spirit of the previous remark. Suppose that the L\'{e}vy measure $\nu$ is bounded away from zero, i.e.\ there exists some $\epsilon>0$ such that $\supp(\nu) \cap [0,\epsilon) = \emptyset$; in particular, $\nu$ is a finite measure. In this case, the series expansion in \eqref{eq-landriault} becomes a closed-form expression since $\supp(\nu^{*n})\cap[0,n\epsilon) = \emptyset$ for any $n\in\nat$. As usual, we use $\lfloor x \rfloor := \max \{y\in\integer\::\: y\leq x \}$ to denote the largest integer below $x\in\real$.

\begin{example}[Geometrically distributed jumps]
	Let $\nu= \sum_{k=1}^\infty (1-p)^{k-1}p\, \delta_{k}$ be a geometric distribution with success probability $p\in (0,1)$. It is well known that the $n$-fold convolution $\nu^{*n}$ is a negative binomial distribution, i.e.
	\begin{align*}
		\nu^{\ast n}=\sum_{k=n}^\infty \binom{k-1}{k-n} (1-p)^{k-n}p^k\, \delta_k, \quad n\in \nat.
	\end{align*}
	Hence for $q\geq 0$ the $q$-scale function of the spectrally negative L\'{e}vy process with Laplace exponent
	\begin{gather*}
		\psi(\beta)=c'\beta+\sigma^2\beta^2+\int_{ (0,\infty)} \left(\ee^{-\beta x}-1\right) \nu(\diff x)
	\end{gather*}
	and $\sigma^2>0$, is given by
	\begin{align*}
		W^{(q)}(x)
		= \sum_{n=0}^{\lfloor x\rfloor}\sum_{k=n}^{\lfloor x\rfloor} (-1)^n
            \frac{ \binom{k-1}{k-n} (1-p)^{k-n}p^k}{(\sigma^2)^{n+1}}\left(\frac{x^{n}}{n!}\ee^{-a_3 x}\right)*\left(\frac{x^{n}}{n!}\ee^{-a_4 x}\right)(x-k),\; x\geq 0,
	\end{align*}
	with $a_3,a_4$ as in Remark~\ref{rem-reform}.
\end{example}

\begin{example}[Zero-truncated Poisson distributed jumps]
	Let $\nu$ denote the distribution of a Poisson random variable $\xi$ with parameter $\mu>0$.  In order to determine the $q$-scale functions using Remark~\ref{rem-reform}, we need to know the probability mass functions $f_n$ of $\nu^{*n}$, $n\in\nat$. For $k\neq 0$ they are given by
	\begin{align}\label{ztp}
		f_n(k)
		= \frac{\PP\left(\sum_{i=1}^n \xi_i = k \:\&\: \forall i \leq n : \xi_i>0 \right)}{\PP\left(\forall i \leq n : \xi_i >0\right)}
		=: \frac{z(k,n)}{(1-\ee^{-\mu})^n} , \quad k\in\nat\setminus\{0\},
	\end{align}
	where $\xi_i$, $i\in\nat\setminus\{0\}$ are i.i.d.\ copies of $\xi$. For $n=0$ we have, by definition, $\nu^{*0}=\delta_0$. The numerator $z(k,n)$ in \eqref{ztp} can be computed recursively: obviously, we have $z(k,n)=0$, if $k<n\in\nat\setminus\{0\}$, and $z(k,1)=\frac{\mu^s}{s!}\ee^{-\mu}$, if $k\in\nat\setminus\{0\}$. Otherwise, we have
	\begin{align*}
		z(k,n)
		&=\PP\left(\sum_{i=1}^n  \xi_i = k\right) - \sum_{\ell=1}^{n-k-1} \binom{n}{\ell} \PP\left(\sum_{i=1}^{\ell} \xi_i =0\right)z(k,n-\ell)\\
		&= \frac{(n\mu)^k}{k!}\ee^{-n\mu} - \sum_{\ell=1}^{n-k-1} \binom{n}{\ell} \ee^{-\ell\mu}z(k,n-\ell).
	\end{align*}
	Setting $z(0,0)=1$ and $z(k,0)=0$, $k\in\nat\setminus\{0\}$, the  $q$-scale function of the spectrally negative L\'{e}vy process with Laplace exponent
	\begin{gather*}
		\psi(\beta)=c'\beta+\sigma^2\beta^2+\int_{(0,\infty)} \left(\ee^{-\beta x}-1\right) \nu(\diff x)
	\end{gather*}
	and $\sigma^2>0$, is given by
	\begin{gather*}
		W^{(q)}(x)
		= \frac{1}{c} \sum_{n=0}^{\lfloor x \rfloor} \sum_{k=n}^{\lfloor x \rfloor}\frac{(-1)^n z(n,k)}{(\sigma^2)^{n+1}\left(1-\ee^{-\mu}\right)^k}\left(\frac{x^{n}}{n!}\ee^{-a_3 x}\right)*\left(\frac{x^{n}}{n!}\ee^{-a_4 x}\right)(x-k),\qquad x\geq 0,
	\end{gather*}
	with $a_3,a_4$ as in Remark~\ref{rem-reform}.
\end{example}

\subsection{L\'evy processes with paths of bounded variation}\label{S4b}

If $(L_t)_{t\geq 0}$ has paths of bounded variation, the following series expansions for the $q$-scale functions hold true.

\begin {theorem}\label{thm-bounded}
    Let $(L_t)_{t\geq 0}$ be a spectrally negative L\'{e}vy process with triplet $(c,0,\widetilde\nu )$ such that $\int_{(0,1)} x\,\nu(\diff x)<\infty$. Recall $c'=c+\int_{(0,1)} y\,\nu(\diff y)$. For $q\geq 0$ set
    \begin{gather*}
    f(x) :=  q + \bar\nu(x), \quad x\geq0,
    \end{gather*}
     then the $q$-scale function of $(L_t)_{t\geq 0}$ is  given by
	\begin{align}\label{case-bounded}
		W^{(q)}=\bone*\sum_{n=0}^\infty\frac{ f^{*n}}{\left(c'\right)^{n+1}},
	\end{align}
	and the series on the right-hand side converges uniformly on compact subsets of $\real_+$.
\end {theorem}
\begin{proof}
    By Lemma~\ref{lem-ExpRep} \ref{lem-ExpRep-i} the Laplace exponent of $L_t$ is given by $\psi(\beta)=c'\beta-\beta\cL\bar\nu(\beta)$ with $c'=c + \int_{(0,1)} y\,\nu(\diff y)$, and we obtain for $q\geq 0$
	\begin{align*}
		\cL[W^{(q)}](\beta)
		= \frac1{c'\beta -\beta\cL\bar\nu(\beta)-q}
		= \frac1{\beta\left(c'-\cL\bar\nu(\beta)-q\beta^{-1}\right)}.
	\end{align*}
	As in the proof of Theorem~\ref{thm-brownian}, we have $c'>\left|\cL\bar\nu(\beta)-q\beta^{-1}\right|$ for sufficiently large $\beta> 1$, and thus
	\begin{align*}
		\cL[W^{(q)}](\beta)
		&= \frac1\beta\sum_{n=0}^\infty \frac{\left(\cL\bar\nu(\beta)+q\beta^{-1}\right)^n}{(c')^{n+1}}
		= \frac1\beta\sum_{n=0}^\infty \frac{\left(\cL f(\beta)\right)^n}{(c')^{n+1}}
	\end{align*}
    with $f(x) = q + \bar\nu(x)$. The rest of the proof coincides with the proof of Theorem~\ref{thm-brownian} with the slight difference  that we now use
	\begin{align*}
		\left|\left (\bone * \sum_{n=1}^\infty \frac{f^{*n}}{(c')^{n+1}}\right)(x)\right| \leq
		\left\|\sum_{n=1}^\infty \frac{f^{*n}}{(c')^{n+1}}\right\|_{L^1[0,x]}, \quad x\geq0,
	\end{align*}	
	to obtain uniform convergence on compact subsets of $\real_+$.
\end{proof}

\begin{remark}\label{rem-pk}
    Formula \eqref{case-bounded} may be regarded as a generalisation of the Pol\-la\-czeck--Khintchine representation of ruin probabilities in the Cram\'er--Lundberg  compound Poisson model. Let $(L_t)_{t\geq 0}$ be a spectrally negative compound Poisson  process with positive drift, i.e.\ its Laplace exponent is of the form
	\begin{gather*}
		\psi(\beta) = c\beta + \lambda\int_{(0,\infty)} \left(\ee^{-\beta x} -1\right) \Pi(\diff x),
	\end{gather*}
	where $c,\lambda>0$ and $\Pi((0,\infty))=1$. Under the net-profit condition
	\begin{gather*}
        \psi'(0+) = c- \lambda \int_{(0,\infty)} x\,\Pi(\diff x) >0,
    \end{gather*}
	the ruin probability $r(x) :=\PP(\exists t\geq~0\::\: x+L_t < 0)$ is given by the Pollaczeck-Khintchine formula
	\begin{align}\label{pk}
		 r(x) = 1-\psi'(0+)\int_{(0,x)} \sum_{n=0}^\infty\frac{\lambda^n\bar\Pi^{*n}(y)}{c^{n+1}}\,\diff y, \quad x\geq 0,
	\end{align}
    with $\bar{\Pi}(y)=\int_{[y,\infty)} \Pi( \diff x)$, see, e.g.\ \cite[Eq.~IV(2.2)]{asmus10} or \cite[Eq.~(1.15)]{kypri14}. Moreover, we have, cf.\ \cite[Thm.~8.1]{kypri14},
	\begin{gather*}
 r(x)  = 1-\psi'(0+) W^{(0)}(x),
	\end{gather*}
	which yields the following representation of the $0$-scale function:
	\begin{gather*}
		W^{(0)}(x)
		= \int_{(0,x)} \sum_{n=0}^\infty\frac{\lambda^n\bar\Pi^{*n}(y)}{c^{n+1}}\,\diff y
		= \bone * \sum_{n=0}^\infty\frac{\lambda^n\bar\Pi^{*n}(x)}{c^{n+1}}.
	\end{gather*}
    Theorem~\ref{case-bounded} works, however, both for all $q\geq 0$ and for a larger class of jump measures, and it does not require any assumption on $\psi'(0+)$.

    In the actuarial context, and in renewal theory, one often uses the following definition of the convolution
	\begin{align*}
		f\circledast g(x)
		&:= \int_{[0,x]} f(x-y)\,\diff g(y),\\ \quad
		f^{\circledast n}(x) &:= \int_{[0,x]} f^{\circledast (n-1)}(x-y)\,\diff f(y),
		\quad\text{and}\quad
		f^{\circledast 0} := \bone,
	\end{align*}
	for any $n\in\nat$ and whenever these integrals are defined. In particular,  $f\circledast g = f*g'$ if $g$ is absolutely continuous with derivative $g'$. With this notation, \eqref{pk} can be rewritten in the more familiar form
	\begin{align*}
		r(x)=1-\psi'(0+)\sum_{n=0}^\infty\frac{(\lambda\mu)^n \Pi_I^{\circledast n}(x)}{c^{n+1}} = 1-\left(1-\frac{\lambda \mu}{c} \right)\sum_{n=0}^\infty\left(\frac{\lambda\mu}{c}\right)^n \Pi_I^{\circledast n}(x)
	\end{align*}
    with $\Pi_I(x):=\mu^{-1}\int_{(0,x)} \bar\Pi(y)\,\diff y  = \mu^{-1}\bone*\bar\Pi(x)$ and the first moment $\mu  =\int_{ (0,\infty)} x\,\Pi(\diff x)$ of $\Pi$.
\end{remark}

\subsection{L\'evy processes with paths of unbounded variation without a Gaussian component}\label{S4c}

It remains to consider spectrally negative L\'{e}vy processes with paths of unbounded variation and $\sigma^2=0$. The following theorem provides a series expansion of the $q$-scale involving an auxiliary function $h$   such that $h*\bar{\bar\nu}(0+)=1$.   The existence of $h$ as well as its properties and some examples will be discussed  below.   Since we are only interested in the behaviour of $h$ near $x=0$, only, $h$ need not be unique. Indeed, any $h_a:= h \cdot \bone_{[0,a]}$, with $a>0$, will be as good as $h$.  

\begin {theorem} \label{thm-unbdd}
    Let $(L_t)_{t\geq 0}$ be a spectrally negative L\'{e}vy process with triplet $(c,0,\widetilde\nu)$ such that $\int_{(0,1]} y\,\nu(\diff y) = \infty$ and $\int_{[1,\infty)} y\,\nu(\diff y) < \infty$. Recall that $c''=c-\int_{[1,\infty)} y\,\nu(\diff y)$. If there exists some $h\in L^1_{\loc}(\real_+)$ such that $\partial (h\ast \bar{\bar\nu}) \in L^1_{\loc}(\real_+)$ and $h*\bar{\bar\nu}(0+)=1$, then the $q$-scale function of $(L_t)_{t\geq 0}$ is given by
	\begin{align}\label{eq-h_series}
		W^{(q)}=H*\sum_{n=0}^\infty f^{*n},\quad q\geq 0.
	\end{align}
    Here, $H(x) :=  \int_{(0,x)} h(y)\,\diff y = \bone\ast h(x)$ is the primitive of $h$, and
    \begin{gather*}
        f(x) := qH(x)- c''h(x)-\partial(h*\bar{\bar\nu})(x),\quad x\geq0.
    \end{gather*}
    The series on the right-hand side of \eqref{eq-h_series} converges uniformly on compact subsets of $\real_+$.
\end {theorem}
\begin{proof}
    Let $h$ be the function from the statement of the theorem   and assume, in addition, that $h\in L^1(\real_+)$ which ensures that the  Laplace transform exists. The identity \eqref{eq-LapDif} for the Laplace transforms shows
	\begin{align*}
		\cL[\partial(h*\bar{\bar\nu})](\beta)
         = \beta\cL h(\beta) \cL\bar{\bar\nu}(\beta)- h*\bar{\bar\nu}(0+)
        = \beta\cL h(\beta) \cL\bar{\bar\nu}(\beta)-1.
	\end{align*}
    By assumption, $\nu_h:=\partial(h*\bar{\bar\nu})$ exists. Thus, the Laplace exponent $\psi$ can be written as
	\begin{gather*}
		\psi(\beta) = c''\beta + \frac\beta{\cL h(\beta)}\left( \cL\nu_h(\beta) +1\right)
	\end{gather*}
	and, from the definition of the scale function, we get for all $q\geq 0$
	\begin{align*}
		\cL[W^{(q)}](\beta)
		= \frac{\cL h(\beta)}\beta\frac1{\left( c''\cL h(\beta) + \cL\nu_h(\beta) + 1 - q\beta^{-1}\cL h(\beta) \right)}.
	\end{align*}
	Since $\nu_h, h\in L^1_{\loc}(\real_+)$ we obtain
	\begin{align*}
		\cL[W^{(q)}](\beta)=\frac{\cL h(\beta)}{\beta}\sum_{n=0}^\infty \left(\cL f(\beta)\right)^n
	\end{align*}
    for large enough $\beta> 1$ and with $f(x) = q\int_{ (0,x)} h(y)\,\diff y - c'' h(x) - \partial(h*\bar{\bar\nu})(x)$ as in Theorem~\ref{thm-unbdd}.   Using the same arguments as in the proof of Theorem~\ref{thm-brownian} we can invert the Laplace transform and obtain formula \eqref{eq-h_series}.

  If $h\in L^1_\loc(\real+)\setminus L^1(\real_+)$, we set $ h_a   =h\cdot\bone_{[0,a]}$ for some $a>0$ and observe  that  $h_a$   meets the requirements stated in the theorem and that $\cL( h_a  )$ is well-defined. Thus, $W^{(q)}$ is given by the series expansion \eqref{eq-h_series} w.r.t.\  $h_a$ .  Changing $h$ for  $h_a$    does not affect $f(x)$, hence $W^{(q)}(x)$ if $x\in [0,a]$. Since $a>0$ is arbitrary, the proof is complete.
\end{proof}

\begin{remark}
	Note that the same reasoning as explained in Remark \ref{rem-addcond} allows us to omit the moment condition $\int_{[1,\infty)} y\,\nu(\diff y) < \infty$ in Theorem \ref{thm-unbdd}.
\end{remark}

\subsubsection*{Resolvents, Bernstein functions and renewal equations}

In order to find a suitable function $h$ in the setting of Theorem \ref{thm-unbdd}, i.e.\ $h\in L^1_{\loc}(\real_+)$ such that $\partial (h\ast \bar{\bar\nu}) \in L^1_{\loc}(\real_+)$ and $h*\bar{\bar\nu}(0+)=1$, let us consider the auxiliary problem
\begin{gather}\label{resolvent}
    \rho * \bar{\bar\nu} \equiv 1,
\end{gather}
where $\rho = \rho(\diff x)$ is a measure on $[0,\infty)$ and $\bar{\bar\nu}$ is the integrated tail function. Equations of this type are also used in the theory of Volterra integral equations where the measure $\rho$ appearing in \eqref{resolvent} is called the resolvent of the function $\bar{\bar\nu}$, cf.\ \cite{gripe10}. Taking Laplace transforms on both sides of \eqref{resolvent} yields
\begin{gather}\label{lap-resolvent}
    \cL[\rho](\beta) = \frac{1}{\beta\cL[\bar{\bar\nu}](\beta)},\quad \beta > 0.
\end{gather}

The calculations in the proof of Lemma~\ref{lem-ExpRep} in our current setting show that $\beta\cL[\bar{\bar\nu}](\beta)$ $=\int_{(0,\infty)} \left(1-\ee^{-x\beta}\right)\bar\nu(x)\,\diff x$. Hence $\beta\cL[\bar{\bar\nu}](\beta)$ is a Bernstein function ($\BF$) with   L\'evy   triplet $(0,0,\bar\nu(x)\,\diff x)$, cf.\ \cite[Thm.~3.2]{schil12}. Thus, its  reciprocal ${1}/(\beta\cL[\bar{\bar\nu}](\beta))$ is a potential, cf.\ \cite[Def.~5.24]{schil12}, and in particular, it is a completely monotone function. From Bernstein's theorem, cf.\ \cite[Thm.~1.4]{schil12}, we thus see that \eqref{lap-resolvent}, hence \eqref{resolvent}, \emph{always} has a \emph{measure-valued} solution $\rho$, see also \cite{gripe80}.
  These observations also allow for an important probabilistic interpretation of \eqref{resolvent}:  $\rho \ast \bar{\bar{\nu}}=1$ is the probability that a subordinator $(S_t)_{t\geq 0}$ with BF $\beta\cL[\bar{\bar\nu}](\beta)$ and L\'evy triplet $(0,0,\bar\nu(x)\,\diff x)$ crosses any level $x>0$, see Bertoin \cite[Chapter III.2]{berto98}. The measure $\rho$ is then the potential or zero-resolvent of $(S_t)_{t\geq 0}$.

For our purpose, however, only $L^1_\loc(\real_+)$-solutions of \eqref{resolvent} are of interest. The following lemma constructs such solutions in the setting of Theorem \ref{thm-unbdd}.

\begin{lemma}\label{lem-density}
    Let $\nu$ be as in Theorem~\ref{thm-unbdd}. Then the convolution equation $\rho * \bar{\bar\nu}(x) = 1$, $x\geq 0$, has a unique strictly positive solution $\rho\in L^1_\loc(\real_+)$.
\end{lemma}

    Using the probabilistic interpretation of $\rho*\bar{\bar\nu}$, the fact that $\rho$ is absolutely continuous is not surprising. One can argue like this: Since the L\'evy measure $\bar\nu(x)\,\diff x$ of the subordinator $(S_t)_{t\geq 0}$ is absolutely continuous, this property is inherited by the transition densities $\PP(S_t\in dx)$, cf.\ Sato \cite[Thm.~27.7]{sato13} and all $q$-resolvent measures
    \cite[
    p. 242]{sato13}, in particular by the zero-resolvent measure $\rho$. The following proof, based on the renewal equation, gives a short and direct approach to this, complementing also Neveu's renewal theorem, cf.\ \cite[Prop.~1.7]{berto99}.

\begin{proof}[  Proof of Lemma~\ref{lem-density}]
    Set $\phi(\beta) := \int_{(0,\infty)} \left(1-e^{-x\beta}\right)\bar\nu(x)\,\diff x$ and $\bar N(x) = x\bar\nu(x)$.
    Under the assumptions of Theorem~\ref{thm-unbdd} we know that $\bar{\bar\nu}(0+)=\infty$.
     We have already seen, that there exists a measure $\rho$ on $[0,\infty)$ solving \eqref{resolvent}. Since $\bar{\bar\nu}(0+)=\infty$, the measure $\rho$ cannot have atoms: indeed, if $\rho \geq c\delta_{y}$ for some $c>0$ and $y\in [0,\infty)$, this would lead to the contradiction
    \begin{gather*}
    	1 = \rho*\bar{\bar\nu}(x) \geq c\delta_y*\bar{\bar\nu}(x) = c\bar{\bar\nu}(x-y) \xrightarrow{x\downarrow y} c\bar{\bar\nu}(0+) = \infty.
    \end{gather*}
    Moreover, $\rho$ satisfies $\cL[\rho](\beta) = (\beta\cL[\bar{\bar\nu}](\beta))^{-1} = 1/\phi(\beta)$. Therefore,
    \begin{gather*}
        \partial\cL[\rho] = \partial\frac 1{\phi} = -\frac 1{\phi^2}\partial\phi = -\cL[\rho]\cL[\rho]\cL[y\bar\nu(y)]
        = -\cL[\rho*\rho*\bar N].
    \end{gather*}
    On the other hand, $\partial\cL[\rho] = -\cL[y\rho]$, and the uniqueness of the Laplace transform shows that
    \begin{gather}\label{eq-double-conv}
        x\rho(\diff x) = \rho*\rho*\bar N(x)\,\diff x.
    \end{gather}
    This implies $\rho(\diff x) = \rho\{0\}\delta_0(\diff x) + r(x)\,\diff x$ with the density function $r(x) =  x^{-1} (\rho *\rho*\bar N)(x)$ on $(0,\infty)$. Because of $\bar{\bar\nu}(0+)=\infty$ we know that $\rho\{0\}=0$.

    Further, recall that $\supp(\rho*\sigma)=\overline{\supp\rho+\supp\sigma}$ holds for any two measures $\rho, \sigma$. Thus we see from $\rho*\bar{\bar\nu}=1$ that $0\in\supp\bar{\bar\nu}\cap\supp\rho$; in particular, $\supp\rho$ and $\supp\bar{\bar\nu}$ contain some neighbourhood of $0$. Using again \eqref{eq-double-conv}, $\supp\rho = \overline{\supp\rho+\supp\rho+\supp\bar N}$ and $0\in\supp\bar N$ (this follows from $0\in\supp\bar{\bar\nu}$), we conclude that $\supp\rho$ is unbounded and even $\supp\rho=[0,\infty)$, since it contains a neighbourhood of $0$.

Finally, by the very definition of the convolution,
\begin{gather*}
	r(x)
	= \frac 1x \int_{(0,x)}\int_{(0,x-y)} \bar N(x-y-z)\,\rho(\diff z)\,\rho(\diff y),
\end{gather*}
proving that $r(x)>0$ since $\supp\rho=[0,\infty)$ and $0\leq \bar N\not\equiv 0$. Thus $\rho$ is positive as well.
\end{proof}

The proof of further regularity properties of solutions to \eqref{resolvent} seems to be quite difficult, with the exception of complete monotonicity ($\CM$) treated in the following corollary. Although this result is known, cf.\ \cite{rao06} or \cite[Rem 2.2]{song06}, we provide a short alternative proof for the readers convenience.

\begin{corollary}\label{cor-cm-density}
    Let $\nu$ be as in Theorem~\ref{thm-unbdd}  and assume that $\nu(\diff z) = n(z)\,\diff z$ with a completely monotone function $n\in \CM$. Then the convolution equation $\rho * \bar{\bar\nu}(x) = 1$, $x\geq 0$ has a unique solution $\rho\in \CM$. Conversely, if the solution of $\rho*\bar{\bar\nu}=1$ is completely monotone, so is $\bar{\bar\nu}$, hence $n$.
\end{corollary}

\begin{proof}
    From $\bar{\bar\nu}(x) = \int_{[x,\infty)}\int_{[y,\infty)} n(z)\,\diff z\,\diff y$ we see that $n\in \CM$ if, and only if, $\bar{\bar\nu}\in \CM$. Moreover, $\cL\bar{\bar\nu}$ is a Stieltjes transform (i.e.\ a double Laplace transform) and $\beta\cL\bar{\bar\nu}(\beta)$ a complete Bernstein function, see \cite[Def.~2.1, Thm.~6.2]{schil12}. Further, the reciprocals of complete Bernstein functions are Stieltjes functions \cite[Thm.~7.3]{schil12}, i.e.\ $\rho(\diff y) = \rho\{0\}\delta_0 + r(y)\,\diff y$ for some completely monotone function $r$. As before, the condition $\bar{\bar\nu}(0+)=\infty$ ensures that $\rho\{0\}=0$.
    The converse statement follows from the symmetric roles played by $\rho$ and $\bar{\bar\nu}$ in the equation $\rho*\bar{\bar\nu}=1$.
\end{proof}

\begin{remark}\label{rem-density}
        With the exception of Corollary~\ref{cor-cm-density}, higher-order regularity properties of the solution to $\rho*\bar{\bar\nu}\equiv 1$ seem to be quite difficult to prove, see also \cite{gripe78} for a related study.
    The following observation points in an interesting direction. Recall that a Bernstein function $f\in \BF$ is a special Bernstein function (see \cite[Chapter 11]{schil12}), if the conjugate function $\beta/f(\beta)$ is also a Bernstein function; iterating this argument, it becomes `self-improving' and shows that $\beta/f(\beta)$ is even a special Bernstein function. A subordinator whose Laplace exponent is a special Bernstein function is a special subordinator. Note that complete Bernstein functions are special, and we have used this (implicitly) in Corollary~\ref{cor-cm-density}.

    Assume, for a moment, that we know $\beta\cL[\rho](\beta)\in \BF$. Then  $\cL[\bar{\bar\nu}](\beta)=1/(\beta\cL[\rho](\beta))$ is the potential of a special subordinator, hence $\beta\cL[\rho](\beta)$ is a special Bernstein function. Thus $\beta/\beta\cL[\rho](\beta) = 1/\cL[\rho](\beta) = \beta\cL[\bar{\bar\nu}](\beta)$ is also a special Bernstein function. This, in turn, means that $\cL[\rho](\beta) = 1/\beta\cL[\bar{\bar\nu}](\beta)$ is the potential of a special subordinator and according to \cite[Thm.~11.3]{schil12}, $\rho$ is of the form $\rho(\diff x) = c\delta_0(\diff x) + r(x)\,\diff x$ for some constant $c\geq 0$ and a decreasing function $r(x)$. If we assume, as before, that $\nu(0,\infty)=\infty$, $\rho$ has no atoms, so $c=0$ and $\rho(\diff x) = r(x) \,\diff x$.

	From $\int_{(0,\infty)} (1-e^{-\beta x}) \mu(\diff x) = \beta \cL[\bar{\mu}](\beta)$ we see that $\beta\cL[\rho](\beta)$ is a Bernstein function, only if $\rho$ has a decreasing density. This means, the above argument becomes a vicious circle, unless we can come up with an independent criterion: let us consider the case that $x\mapsto\bar\nu(x)$ is logarithmically convex,   i.e.
    \begin{gather*}
        \bar\nu(\lambda x + (1-\lambda)y)\leq \bar\nu(x)^\lambda\bar\nu(y)^{1-\lambda}
        \quad\text{for all $x,y>0$ and $\lambda \in (0,1)$}.
    \end{gather*}
    Since $\bar{\bar\nu}(x)$ has a  negative  derivative, $\beta\cL[\bar{\bar\nu}](\beta)=\int_{(0,\infty)} \left(1-e^{-\beta x}\right)\bar\nu(x)\,\diff x$ is a Bernstein function and, because of log-convexity, even a special Bernstein function (see \cite[Thm.~11.11]{schil12}), hence the conjugate function $1/\cL[\bar{\bar\nu}](\beta) = \beta\cL[\rho](\beta)$ is a Bernstein function. Together with the above consideration this proves:

    \medskip\noindent
    {\itshape If $\nu(0,\infty)=\infty$ and $x\mapsto\bar\nu(x)$ is logarithmically convex, the solution to $\rho*\bar{\bar\nu}\equiv 1$ is of the form $\rho(\diff x) = r(x)\,\diff x$ with a decreasing density $r(x)>0$ on $(0,\infty)$.}\medskip

    See also \cite[Lemma 2.2]{kypri10} for a similar result stated in terms of potential measures of subordinators.
\end{remark}

The following lemma provides a particular representation of the resolvent of $\bar{\bar\nu}$ as a continuous function; in fact, it is the
Radon-Nikodym derivative of the scale function of some L\'{e}vy process.

\begin{lemma}\label{lem-compensated}
    Let $(L_t)_{t\geq 0}$ be a spectrally negative L\'{e}vy process with triplet $(c,0,\widetilde\nu)$ such that $\int_{(0,1)} y\,\nu(\diff y) = \infty$ and $\int_{[1,\infty)} y\,\nu(\diff y) < \infty$. Denote by $\widetilde L_t:= L_t-\EE[L_t]$ the compensated process, and write $\widetilde W$ for the $0$-scale function of $(\widetilde L_t)_{t\geq 0}$. Then $\widetilde W' * \bar{\bar\nu}= \bone$, i.e.\ $\rho:=\widetilde W'$ is a resolvent of $\bar{\bar\nu}$. In particular, $\rho\in\cC(0,\infty)$.
\end{lemma}	
\begin{proof}
	From Lemma~\ref{lem-ExpRep} \ref{lem-ExpRep-ii} we know that the Laplace exponent of $(\widetilde{L}_t)_{t\geq 0}$ is given by
	\begin{gather*}
		\widetilde{\psi}(\beta)
		= \psi(\beta) - \EE[L_1] \beta
		= c''\beta + \beta^2 \cL[\bar{\bar\nu}](\beta) - c'' \beta
		= \beta^2 \cL[\bar{\bar\nu}](\beta).
	\end{gather*}
	By the definition of the scale function, we see $\cL \widetilde{W} = (\beta^2\cL\bar{\bar\nu})^{-1}$, and with \eqref{eq-LapDif} it follows that
	\begin{align*}
		\cL [\widetilde{W}'](\beta) = \frac 1{\beta\cL[\bar{\bar\nu}](\beta)}
		\iff \cL [ \widetilde{W}'](\beta)\cdot \cL[\bar{\bar\nu}](\beta) = \frac 1{\beta},
	\end{align*}
    since $\widetilde{W}(0 +)=0$. Finally, by \cite[Lemma 2.4]{kuzne13} we have $\widetilde W\in \cC^1(0,\infty)$, and thus $\rho\in\cC(0,\infty)$.
\end{proof}

The above considerations prove that in the setting of Theorem \ref{thm-unbdd} there always exists a resolvent of the integrated tail function $\bar{\bar\nu}$. This immediately leads to the following simple corollary of Theorem \ref{thm-unbdd}.

\begin{corollary}\label{cor-unbdd}
    Let $(L_t)_{t\geq 0}$ be as in Theorem~\ref{thm-unbdd} and let $\rho\in L^1_{\loc}(\real_+)$ be a resolvent of $\bar{\bar\nu}$. Then the $q$-scale function of $(L_t)_{t\geq 0}$ is given by
	\begin{align}\label{eq-rho_series}
		W^{(q)} = R*\sum_{n=0}^\infty (qR-c''\rho)^{*n},\quad q\geq 0,
	\end{align}
    where $R(x) := \int_0^x \rho(y)\,\diff y = \bone\ast\rho (x)$ for all $x\in\real_+$. The series on the right-hand side of \eqref{eq-rho_series} converges uniformly on compact subsets of $\real_+$.
\end{corollary}

	The deep relationship between scale functions and renewal equations has been explained in detail in \cite{chan11}. From the definition of scale functions \eqref{def-scalefunction} one immediately obtains that for a spectrally negative L\'{e}vy process $(L_t)_{t\geq 0}$ with triplet $(c,2\sigma^2,\widetilde\nu)$ and $\int_{[1,\infty)}y\nu(\diff y)<\infty$
\begin{align}\label{eq_scale_eq}
c''W^{(0)} + \sigma^2 \partial W^{(0)} + \bar{\bar\nu}*\partial W^{(0)} = 1.
\end{align}
	In the case \(\sigma^2>0\) and with the notation of Remark \ref{rem-pk} this equation is a renewal equation on \(\real_+\) of the type
	\begin{align}\label{eq-renewal_1}
    f=1+f\circledast g
	\end{align}
	with \(f(x)=\sigma^2\partial W^{(0)}(x)\) and \(g(x)= -2\sigma^{-2} \int_0^x\bar{\bar\nu}(y)\diff y - 2c''\sigma^{-2}x\). This has been dicussed in \cite{chan11} where \eqref{eq-renewal_1} is solved for \(\partial ^{(0)}\) to infer smoothness properties of \(W^{(0)}\). However, if \(\sigma^2=0\) and \(\int_{(0,1)}y\nu(\diff y)=\infty\) then \eqref{eq_scale_eq} yields
	\begin{align*}
	c''W^{(0)} = 1 - \frac{\bar{\bar\nu}}{c''}*\partial (c''W^{(0)}) = 1 - \frac{\bar{\bar\nu}}{c''} \circledast (c''W^{(0)})
	\end{align*}
	which is not a renewal equation in the sense of \cite{chan11} since \(\bar{\bar\nu}(0+)\neq0\). Nevertheless, our approach to obtain series expansions for \(W^{(0)}\) in such cases still provides a technique to solve this kind of renewal equations on $\real_+$, namely
\begin{align}\label{eq-renewal}
    f=1+g*\partial f =1+ g\circledast f
\end{align}
where we do not necessarily assume that $g(0+)=0$. This is illustrated in the next theorem, whose proof follows the lines of the proof of Theorem \ref{thm-unbdd}.

\begin {theorem}\label{ren-eq}
    Let $g\in L^1_\loc(\real_+)$.
    \begin{enumerate}
    \item\label{ren-eq-i}
        If $g$ admits a resolvent $\rho$, then \eqref{eq-renewal} is solved by $f=\bone*\sum_{n=1}^\infty  (-1)^n \rho^{*n}$.
    \item\label{ren-eq-ii}
        If there exists a  \textup{(}not necessarily unique\textup{)}   function $h\in L^1_{loc}(\real_+)$ with $g~*~h (0+)=1$ and such that $g_h := \partial(g*h)$ is well-defined, then \eqref{eq-renewal} is solved by $$f = \bone*h*\sum_{n=1}^\infty (-1)^n(h+ g_h)^{*(n-1)}.$$
    \end{enumerate}
\end {theorem}
\begin{proof}
It is enough to prove \ref{ren-eq-ii}. First note that $f(0+)=0$ since $f(x)=\int_{[0,x]}u(y)\,\diff y$ with locally integrable integrand $u(y):=h*\sum_{n=1}^\infty (-1)^n(h+ g_h)^{*(n-1)}$. Convolving both sides of \eqref{eq-renewal} with $h$, and applying Lemma \ref{lem-Conv-C1} \ref{lem-Conv-C1-i} yields
\begin{align*}
    h*f
    &=\bone*h+h*g*\partial f = \bone*h+\partial(h*g* f)= \bone*h + \partial(h*g)*f+f.
\end{align*}
Applying the Laplace transform and solving for $\cL[f](\beta)$ we get
\begin{align*}
    \cL[f](\beta)
    = \frac{\cL[\bone*h](\beta)}{\cL[h](\beta) - \cL[g_h](\beta) - 1}
    &= -\cL[\bone*h](\beta) \sum_{n=0}^\infty \left(\cL[h-g_h](\beta)\right)^n\\
    &= -\cL[\bone*h](\beta) \sum_{n=0}^\infty \cL[(h-g_h)^{*n}](\beta),
\end{align*}
from which the assertion follows due to the uniqueness of the Laplace transform.
\end{proof}

\subsubsection*{Examples}

Although the existence of a resolvent $\rho$ in Corollary~\ref{cor-unbdd} is clear, there seems to be no general method of constructing $\rho$ explicitly. Fortunately, there are quite a few cases where a suitable function $h$ (as in Theorem~\ref{thm-unbdd}) can be constructed. In particular, this is true if $\bar{\bar\nu}$ is regularly varying at $0$. Recall that a measurable function $f:(0,\infty)\to\real$ is regularly varying at $0$ with index $\gamma\in \real$, if
\begin{gather*}
    \lim_{x\to   0+}\frac{f(\lambda x)}{f(x)}=\lambda^{-\gamma} \quad
    \text{exists for all }\lambda >0.
\end{gather*}
We write $\mathcal R_\gamma$ for the space of regularly varying functions with index $\gamma$. Our standard reference is the monograph \cite{bingh11}.

 Let $(L_t)_{t\geq 0}$ be a spectrally negative L\'evy process as in Theorem~\ref{thm-unbdd}, and assume that $\bar{\bar\nu}\in\mathcal{R}_\gamma$ for some $\gamma\in(0,1)$. From \cite[Thm.~1.8.3]{bingh11} and its proof we know that there exists a completely monotone function $k$ such that $\lim_{x\to 0} \bar{\bar\nu}(x)/k(x)=~1.$ With the arguments of the previous subsection this implies that there exists a unique completely monotone function $h$  for which  $h*k=1$.
Set $k_\epsilon:=\bar{\bar\nu}-k$. Since   $\lim_{x\to0+}{k_\epsilon(x)}/{k(x)}=0$  , we obtain
\begin{equation}\label{eq-rvhelp1}\begin{aligned}
    |h*k_\epsilon(x)|
    &= \left| \int_{ (0,x)} h(x-y) k_\epsilon (y) \,\diff y \right|\\
    &\leq \int_{ (0,x)} h(x-y)\left|\frac{k_\epsilon(y)}{k(y)}\right|k(y)\,\diff y
    \leq \sup_{y\leq x} \left|\frac{k_\epsilon(y)}{k(y)}\right| \xrightarrow[x\to   0+]{}0,
\end{aligned}\end{equation}
hence,
\begin{equation}\label{eq-rvhelp2}
    h\ast \bar{\bar\nu}(0+)=1.
\end{equation}
We may thus use this function $h$ in Theorem \ref{thm-unbdd} to obtain the representation of the $q$-scale function given in the following corollary.

\begin{corollary}\label{cor-regvar-I}
    Let $(L_t)_{t\geq 0}$ be a spectrally negative L\'evy process as in Theorem~\ref{thm-unbdd} and assume that the L\'evy measure $\nu$ is such that $\bar{\bar\nu}\in\mathcal{R}_\gamma \cap \cC^1(0,\infty)$ for some $\gamma\in(0,1)$.  Let $k_\epsilon$ and $h$ be the functions in \eqref{eq-rvhelp1}. Then
    \begin{align*}
        W^{(q)}=H*\sum_{n=0}^\infty \big(qH-c''h -\partial(h*k_\epsilon)\big)^{*n}
    \end{align*}
    where $H(x) := \int_{(0,x)} h(y)\,\diff y = \bone\ast h (x)$ for all $x\in\real_+$.
\end{corollary}
\begin{proof}
    We check the conditions stated in Theorem \ref{thm-unbdd}. First of all, $h\in L_\loc^1(\real_+)$ as it is the resolvent of $k\in L^1_\loc(\real_+)$. Furthermore, we note that
    \begin{gather*}
        h\ast \bar{\bar\nu} = h\ast (k + k_\epsilon) = 1 + h\ast k_\epsilon,
    \end{gather*}
    which implies $\partial(h\ast \bar{\bar\nu}) = \partial (h\ast k_\epsilon)$. By assumption, $k_\epsilon\in \cC^1(0,\infty)$. Thus, $h*k_\epsilon\in\cC^1(0,\infty)$ by Lemma \ref{lem-Conv-C1} \ref{lem-Conv-C1-ii}, and as $h*k_\epsilon(0+)=0$ we obtain $\partial(h*k_\epsilon)\in L^1_{\loc}(\real_+)$. Finally, by \eqref{eq-rvhelp1}, $h\ast \bar{\bar\nu}(0+) = 1+ h\ast k_\epsilon(0+) = 1$ finishing the proof.
\end{proof}

In the following, we restrict ourselves to jump measures $\nu$ satisfying $\bar{\bar\nu}\in\mathcal{R}_\gamma$, $\gamma\in(0,1)$, and for which, in addition, there exists a constant $C>0$ such that
\begin{equation}\label{eq-condfraccalc}
    \lim_{x\to   0+}\frac{ Cx^{-\gamma}}{\bar{\bar\nu}(x)}=1.
\end{equation}
Clearly, these are special cases of the jump measures treated in Corollary \ref{cor-regvar-I}, but we can now use techniques from fractional calculus which will simplify our calculations. Recall that the Riemann--Liouville fractional integral $I_{0+}^\alpha f$ of order $\alpha\in(0,1)$ of a function $f:\real_+\to\real_+$ is defined by
\begin{gather*}
	(I_{0+}^\alpha f)(x)
    :=\frac1{\Gamma(\alpha)} \int_{ (0,x)} \frac{f(y)}{(x-y)^{1-\alpha}} \,\diff y =  \frac1{\Gamma(\alpha)}((\cdot)^{\alpha-1}*f)(x),\quad x\geq 0,
\end{gather*}
and the Riemann--Liouville fractional derivative  $D_{0+}^\alpha$ of order $\alpha\in(0,1)$ of a function $f:\real_+\to\real_+$ is defined by
\begin{gather*}
	(D_{0+}^\alpha f)(x):=\partial (I_{0+}^{1-\alpha} f)(x).
\end{gather*}
A thorough introduction to fractional calculus is the monograph \cite{samko93}.

\begin{corollary}\label{cor_frac}
    Let $(L_t)_{t\geq 0}$ be a spectrally negative L\'evy process as in Theorem~\ref{thm-unbdd} and assume that the L\'evy measure $\nu$ satisfies $\bar{\bar\nu}\in\cC^1(0,\infty)$ and \eqref{eq-condfraccalc} for some $C>0$ and let $\gamma\in(0,1)$. Set $h(x):= x^{\gamma-1}$ and $H(x):=\bone*h(x)= \gamma^{-1}x^{\gamma}$ for $x>0$. The $q$-scale function of $(L_t)_{t\geq 0}$ is given by
	\begin{align}\label{eq-frac}
		W^{(q)}
       = H*\sum_{n=0}^\infty \left(\frac{\sin(\gamma\pi)}{C\pi}\right)^{n+1} f^{*n}, \quad q\geq 0,
	\end{align}
	with $f(x):=qH(x)- c''h(x)- \Gamma(\gamma)D^{1-\gamma}_{0+} \bar{\bar\nu}(x)$ for $x>0$.
\end{corollary}	
\begin{proof}
    Set $k(x)=C x^{-\gamma}$ such that $\lim_{x\to 0} \bar{\bar\nu}(x)/k(x)=~1$, and define $k_\epsilon:=\bar{\bar\nu}-k$. Clearly, $k_\epsilon = o(x^{-\gamma})$. Set $h(x)= x^{\gamma-1}$. The same computation as in \eqref{eq-rvhelp1} yields
	\begin{gather*}
        |h\ast k_\epsilon(x)|\to 0 \quad\text{as }x\to0.
    \end{gather*}
	Moreover,
	\begin{align*}
		h\ast k(x)
        &= \int_{(0,x)} h(x-y) k(y) \,\diff y
        = \int_{(0,1)}h(x(1-z)) k(xz) x \,\diff z\\
		&= C\cdot \int_{(0,1)} (1-z)^{\gamma-1} z^{-\gamma} \,\diff z
        = C\cdot B(1-\gamma,\gamma),
	\end{align*}
	where $B(x,y), x,y> 0$, denotes the beta function. Since $B(1-\gamma,\gamma) = {\pi}/{\sin(\gamma \pi)}$, we have
	\begin{gather*}
        h\ast \bar{\bar\nu} (0+)
        = h\ast k_\epsilon(0+) + h\ast k(0+)
        = C\frac{\pi}{\sin(\gamma \pi)}.
    \end{gather*}
    We may thus apply Theorem \ref{thm-unbdd} with $\widetilde{h}(x):=\frac{\sin(\gamma \pi)}{C\pi} h(x)$. Clearly, $\widetilde{h}\in L_\loc^1(\real_+)$ and the calculations above show $\widetilde{h}\ast \bar{\bar\nu} (0+)=1$. Further, $\partial(\widetilde{h}\ast \bar{\bar\nu}) \in L_\loc^1(\real_+)$, since
	\begin{gather*}
        \partial(h\ast \bar{\bar\nu}) (x)
        = \partial \big(\Gamma(\gamma) I_{0+}^{\gamma} \bar{\bar\nu}(x) \big)
        = \Gamma(\gamma)D^{1-\gamma}_{0+} \bar{\bar\nu}(x),
    \end{gather*}
    with $D^{1-\gamma}_{0+} \bar{\bar\nu}\in \cC^\gamma$ as $\bar{\bar\nu}$ in $\cC^1$ by assumption, cf.\ \cite[Lemma 2.2]{samko93}.
    This yields the representation of $W^{(q)}$.
\end{proof}	

\begin{example}[Spectrally negative stable processes]\label{example-3}
    Consider a spectrally negative stable L\'{e}vy process with stability index $\alpha\in(1,2)$ and Laplace exponent $\psi(\beta)=\beta^\alpha$. Then as shown in \cite{berto96}, see also \cite{hubal10}, for all $q\geq 0$
	\begin{align}\label{eq-ex3}
		W^{(q)}=\alpha x^{\alpha-1} E'_\alpha(qx^\alpha),\quad x\geq 0,
	\end{align}
    where $E_\alpha(x)=\sum_{k=0}^\infty {\Gamma(1+\alpha k)}^{-1} x^k$, $x\geq 0$, is the Mittag--Leffler function.

    A lengthy but otherwise straightforward calculation shows that this formula agrees with \eqref{eq-frac}: use that  $\bar{\bar\nu}$ fulfils \eqref{eq-condfraccalc} with $C= 1/{\Gamma(2-\alpha)}$ and $\gamma=\alpha-1$, and moreover  $\int^x_{0+} (x-y)^{\alpha-1}y^{-\alpha}\,\diff y = \pi/\sin(\alpha\pi)$ for any $x>0$ and $\alpha\in(0,1)$.
\end{example}

The following observation allows us to construct further explicit examples of scale functions. Consider a spectrally negative L\'{e}vy process $(L_t)_{t\geq0}$ with Laplace exponent $\psi_L$ and suppose that we know its scale functions $W^{(q)}_L$. Further, let $(C_t)_{t\geq0}$ be an independent compound Poisson process with positive jumps and without drift; we assume that the jumps have intensity $\lambda>0$ and are distributed according to the probability measure $\Pi$. Then, $(L_t-C_t)_{t\geq0}$ is a spectrally negative L\'evy process whose Laplace exponent is given by
\begin{align*}
	\psi_{L-C}(\beta)=\psi_L(\beta)+ \lambda  \cL[\Pi](\beta)-\lambda.
\end{align*}
As for any $\lambda$ we have $|\lambda \cL[\Pi](\beta)\cL[W_L^{(q+\lambda)}(\beta)|<1$ for $\beta$ large enough, the $q$-scale function of $(L_t-C_t)_{t\geq0}$ is thus defined by
\begin{align*}
    \cL[W_{L-C}^{(q)}]
    = \frac1{\psi_{L-C}-q}
    = \frac{\cL[W_L^{(q+\lambda)}]}{1+ \lambda  \cL[\Pi]\cL[W_L^{(q+\lambda)}]}
    = \cL[W_L^{(q+\lambda)}]\sum_{k=0}^\infty\left(- \lambda \cL[\Pi]\cL[W_L^{(q+\lambda)}]\right)^k.
\end{align*}
This immediately implies
\begin{align}\label{eq-independentSum}
	W^{(q)}_{L-C}=W_L^{(q+\lambda)}*\sum_{k=0}^\infty  \lambda^k  \left(-\Pi*W_L^{(q+\lambda)}\right)^{*k}.
\end{align}

\begin{example}[Stable process plus Poisson process]
    Let $(L_t)_{t\geq0}$ be a spectrally negative $\alpha$-stable processes with $\alpha\in(1,2)$ and let $\Pi=\delta_1$. For each $n\in\nat$ and $q\geq 0$ we then have from \eqref{eq-independentSum}
	\begin{align*}
        W^{(q)}_{L-C}(x)
        =W_L^{(q+\lambda)}*\sum_{k=0}^\infty  (-\lambda)^k  \left(W_L^{(q+\lambda)}(\cdot-1)\right)^{*k},
	\end{align*}
	where $W_L^{(q+\lambda)}$ is given by \eqref{eq-ex3}. In fact, the sum on the right-hand side is finite as $W_L^{(q+\lambda)}(x-1)=0$ for $x\in[0,1]$, and thus,
	\begin{gather*}
	   \left(W_L^{(q+\lambda)}(\cdot-1)\right)^{*k}(x)=0\quad\text{for all\ \ } k>x.
	\end{gather*}
\end{example}

\section{Smoothness properties}\label{S5}


The question of smoothness of $q$-scale functions is of considerable practical and theoretical interest. On the one hand, several fluctuation identities require the evaluation of derivatives of the $q$-scale functions, see  e.g.\ \cite[Section~1]{kuzne13}. On the other hand, cf.\ \cite{chan11, kypri10} and \cite[Section~3.5]{kuzne13}, one can interpret $W^{(q)}$ as an eigenfunction of (a suitable extension of) the infinitesimal generator $\mathcal{A}$ of $(L_t)_{t\geq 0}$, i.e.
\begin{gather*}
	(\mathcal{A}-q)W^{(q)} = 0.
\end{gather*}
However, this equation has to be treated with caution, since the domain of $\mathcal{A}$ has to be made explicit.

It has been conjectured by Doney, see \cite[Conjecture~3.13]{kuzne13}, that the smoothness of the scale functions depends on the smoothness of the jump measure. In the three different situations considered in Section~\ref{S4} Doney's conjecture reads as follows:
\begin{enumerate}
\item\label{doney-i}
    If $\sigma^2>0$, then $W^{(q)}\in \cC^{n+3}(0,\infty) \iff \bar{\nu} \in \cC^n(0,\infty)$ for all $n\in\nat, q\geq 0$.

\item\label{doney-ii}
    If $\sigma^2=0$ and $\int_{(0,1)} y\,\nu(\diff y)<\infty$, then  $W^{(q)}\in \cC^{n+1}(0,\infty) \iff \bar{\nu} \in \cC^n(0,\infty)$ for all $n\in\nat, q\geq 0$.

\item\label{doney-iii}
    If $\sigma^2=0$ and $\int_{(0,1)} y\,\nu(\diff y) = \infty$, then $W^{(q)}\in \cC^{n+2}(0,\infty) \iff \bar{\nu} \in \cC^n(0,\infty)$ for all $n\in\nat, q\geq 0$.
\end{enumerate}
Under additional assumptions, there are some partial answers. In \cite[Thm.~2]{chan11}, the equi\-valence in \ref{doney-i} is shown to be true if one assumes, in addition, that the Blumenthal--Getoor index of the pure-jump part satisfies $\inf \left\{ \alpha \geq 0: \int_{(0,1)} x^\alpha \,\nu(\diff x) <\infty \right\} < 2$. The equivalence in \ref{doney-ii} is shown in \cite[Thm.~3]{chan11} under the additional assumption that $\nu(\diff x)=\pi(x)\,\diff x$ such that $\pi(x)\leq C |x|^{-1-\alpha}$ for $\alpha<1$, $C>0$, in a neighbourhood of the origin.
To the best of our knowledge, the only case where \ref{doney-iii} is known to be true is that of a process whose L\'evy measure has a completely monotone density, cf.\ \cite{chan11}. In this section we use the series expansion of Theorem \ref{thm-unbdd} to prove more general smoothness properties for the scale functions of processes in case \ref{doney-iii}.

As one would expect, we have to impose constraints on the behaviour of the function $f$ appearing in Theorem \ref{thm-unbdd} near $x=0$, which resemble those used in the cases \ref{doney-i} and \ref{doney-ii}. More precisely, in order to derive smoothness properties of a convolution power series $\sum_{n=1}^\infty f^{*n}$, one first has to ensure the convergence of the series. This requires, however, that $f$ behaves nicely near the origin. In Theorem \ref{thm-unbdd}, the function $f$ is rather implicit as compared to those in Theorems \ref{thm-brownian} and \ref{thm-bounded} that refer to cases \ref{doney-i} and \ref{doney-ii}, respectively. This is the origin for the comparatively technical conditions in Theorem~\ref{thm-smoothness} below. The, nevertheless, considerable theoretical and practical use of our result will be demonstrated in Section \ref{sub-se-42}.

A key ingredient in the proof is the asymptotic behaviour of a function near the origin. To simplify notation, let us introduce for the function $f\in L^1_{\loc}(\real_+)$ the parameter $\kappa(f)\in \{1,2\}$ defined by
\begin{align*}
    \kappa(f)
    :=
    \begin{cases}
        2,&\text{if there are constants }C,\alpha>0\text{ such that } |f(x)|< Cx^{\alpha-0.5} \text{ for } x\in(0,1), \\
        1,&\text{otherwise}.
\end{cases}
\end{align*}

\begin {theorem}\label{thm-smoothness}
    Let $(L_t)_{t\geq 0}$ be a spectrally negative L\'evy process with triplet $(c,0,\widetilde\nu)$ such that $\int_{(0,1)} y\,\nu(\diff y) = \infty$ and $\int_{[1,\infty)} y\,\nu(\diff y) < \infty$. Let $\bar\nu\in\cC^{k_0}(0,\infty)$ for some $k_0\in\nat$. Further, assume that there exists a decomposition $\bar{\bar\nu}=N_1+N_2$  and some constants $C_1,C_2>0,   \alpha\in(0,1)  $ such that
	\begin{enumerate}
	\item\label{thm-smooth-i}
        $N_1\in L^1_{\loc}(\real_+)$ admits a resolvent $\rho_1\in\cC^1(0,\infty)$ such that $|\rho_1(x)|<C_{1}x^{\alpha-1}$ on $(0,1)$  and
	\item\label{thm-smooth-ii}
        $N_2\in\cC^{k_0+1}(0,\infty)$ with $N_2(0+) \in \real$ and $|\partial N_2(x)|<C_2x^{-\alpha}$ for $ x\in(0,1)$.
	\end{enumerate}	
	Then $W^{(q)}\in\cC^{k_0+\kappa(\rho_1)}(0,\infty)$ for all $q\geq0$.
\end {theorem}

\begin{remark}\label{rem-cond-relax}
In the following we will, without loss of generality, assume $N_2(0+)=0$. Otherwise, if $N_2(0+)=:c_0\neq 0$, we set $d=c''+c_0$ and $\bar{\bar\mu}:=\bar{\bar\nu}-c_0$ such that
\begin{align}\label{eq-cond-relax}
    \psi(\beta)
    = c''\beta+\beta^2\cL[\bar{\bar\nu}]
    = (c''+c_0)\beta+\beta^2\cL[\bar{\bar\nu}-c_0]
    = d\beta+\beta^2\cL[\bar{\bar\mu}]
\end{align}
yields a representation of the Laplace exponent for which the conditions \ref{thm-smooth-i} and \ref{thm-smooth-ii} of Theorem \ref{thm-smoothness} are met. Note that, in general, $\bar{\bar\mu}$ is not an integrated tail function of any spectrally positive L\'{e}vy measure. However, this has no consequences for the proof.
\end{remark}

\subsection{Proof of Theorem \ref{thm-smoothness}}\label{sub-se-41}

We will use the series expansion of Theorem~\ref{thm-unbdd} in order to prove Theorem \ref{thm-smoothness}.
At first sight, this seems rather circumstantial in comparison with the representation in Corollary \ref{cor-unbdd}. The trouble is that for the convolution equation $\rho * f = 1$ rather little is known how the smoothness of $\rho$ and $f$ are related to each other, see Section \ref{S4c} above.

In order to apply Theorem \ref{thm-unbdd}, we need to construct a suitable function $h$. We are going to show that, in general, a suitable approximation of the resolvent of $N_1$  is good enough for our purposes. By assumption, the resolvent $\rho_1$ of $N_1$ is continuously differentiable. The following lemma will help us to approximate $\rho_1'$ and then construct a suitable function $h$ in the subsequent Lemma \ref{lem-h-properties}.

\begin{lemma}\label{lem-contApprox}
    Let $f\in\cC(0,\infty)$. For any strictly positive function $g\in\cC(0,\infty)$ there exists some $u\in\cC^\infty(0,\infty)$ such that $|f(x)-u(x)|<g(x)$ for all $x\in(0,\infty)$.
\end{lemma}

\begin{proof}
    Let $\bigcup_{\lambda}U_\lambda = (0,\infty)$ be any cover of $(0,\infty)$ with open intervals. There is a locally finite, smooth subordinate partition of unity, i.e.\ a sequence $(\chi_n)_{n\geq 1}\subset\cC_c^\infty(0,\infty)$ such that $\supp\chi_n\subset U_{\lambda(n)}$ for some $\lambda(n)$, $\chi_n\geq 0$ and $\sum_{n=1}^\infty\chi_n(x) = \bone_{(0,\infty)}(x)$ such that for every compact set $K$ and $x\in K$ the sum only consists of finitely many summands, see e.g. \cite[Thm.~6.20]{rudin91}.

    Choose $U_n := (2^{n-1},2^{n+1})$, $n\in\integer$. Since $g$ is continuous and strictly positive, $\epsilon_n := \inf_{\overline U_n}g > 0$. By Weierstra{\ss}'s theorem, we can approximate any continuous function $f$ on $\overline U_n$ uniformly by a polynomial $p_n$ such that $\sup_{\overline U_n}|f-p_n|\leq\epsilon_n$. Therefore, we get
    \begin{align*}
        \Big|f - \sum_{n\in\integer}p_n\chi_n\Big|
        = \Big|\sum_{n\in\integer}(f-p_n)\chi_n\Big|
        \leq \sum_{n\in\integer}|f-p_n|\chi_n
        \leq \sum_{n\in\integer}\epsilon_n \chi_n
        \leq \sum_{n\in\integer} g \chi_n
        = g.
    \end{align*}
    Since the sum $\sum_{n\in\integer} \chi_n = \bone_{(0,\infty)}$ is locally finite, we see that $\sum_{n\in\integer} p_n\chi_n\in \cC_c^\infty(0,\infty)$.
\end{proof}

Lemma \ref{lem-contApprox} now leads to a smooth approximation $h$ of $\rho_1$ as follows.

\begin{lemma}\label{lem-h-properties}
Let $\bar{\bar\nu}=N_1+N_2$ be an integrated tail function as in Theorem~\ref{thm-smoothness}  and Remark \ref{rem-cond-relax}, and let $\rho_1$ be the resolvent of $N_1$.   In particular, we require that the assumptions \ref{thm-smooth-i}, \ref{thm-smooth-ii} of Theorem~\ref{thm-smoothness} hold. Then,   there exists a function $h\in\cC^\infty(0,\infty)$ such that
\begin{enumerate}
\item\label{lem-h-i}\ $\lim_{x\downarrow 0} h(x)/\rho_1(x)=1$,
\item\label{lem-h-ii} $h*\bar{\bar\nu}(0+)=1$,
\item\label{lem-h-iii} $\partial(h*\bar{\bar\nu})$ is bounded on $(0,1)$ and
\item\label{lem-h-iv}$h*\partial(h*\bar{\bar\nu})(0+)=0$.
\end{enumerate}
\end{lemma}

\begin{proof}
By assumption, $\rho'_1$ exists and is continuous. Using Lemma~\ref{lem-contApprox} we can find a smooth function $u\in\cC^\infty(0,\infty)$ such that $|u(x)-\rho_1'(x)|\leq x$ for all $x\in (0,1)$.

\medskip\noindent
\ref{lem-h-i}
    By construction, $u-\rho_1'\in L^1_\loc(\real_+)$, and so we may define
    \begin{align*}
        h(x):= - \int_{(x,1)} u(y)\,\diff y  + \int_{(0,1)} (u(y)-\rho_1'(y))\,\diff y + \rho_1(1).
    \end{align*}
    This choice results in
    \begin{align*}
        |h(x)-\rho_1(x)|
        = \left|\int_{(0,x)} \left(u(y)-\rho_1'(y)\right) \diff y\right|
        \leq \int_{(0,x)}  y\,\diff y
        = \frac 12 x^2,\quad x\in (0,1),
    \end{align*}
    which implies that $h$ and $\rho_1$ are asymptotically equivalent at zero, since $\rho_{1}(0+)\neq 0$.

\medskip\noindent
\ref{lem-h-ii} Note that  	
\begin{equation}\label{eq-proofsmooth01}
	\bar{\bar\nu}*h
	= N_1*h+N_2*h
	= N_1*(h-\rho_1)+ N_1*\rho_1 +N_2*h
	= N_1*(h-\rho_1)+ \bone+N_2*h.
\end{equation}
Thus the claim   follows from Lemma \ref{convo-asympt} \ref{convo-asympt-ii} together with the assumption $N_2(0+)=0$.

\medskip\noindent
\ref{lem-h-iii}
    Using dominated convergence we have for the right derivative
	\begin{align*}
    \lefteqn{|\partial_+(N_1*h)(x)| =|\partial_+(N_1*(h-\rho_1))(x)|}\\
        &\leq \left|\lim_{t\to 0}\int_{(0,x)}  \frac{h(x+t-y)-\rho_1(x+t-y)-h(x-y)+\rho_1(x-y)}t N_1(y)\,\diff y\right| \\
        &\quad\mbox{} + \left|\lim_{t\to0}\frac 1t\int_{[x,x+t)}h(x+t-y)-\rho_1(x+t-y)  N_1(y)\,\diff y\right|\\
        &\leq \int_{ (0,x)}  \left|h'(x-y)-\rho_1'(x-y)\right| |N_1(y)| \,\diff y \\ & \quad  + \lim_{t\to 0}\frac 1t\int_{(0,t]} \left|h(y)-\rho_1(y)\right|  |N_1(x+t-y)|\,\diff y. \\
        &\leq x\int_{(0,x)} |N_1(y)|\,\diff y +\lim_{t\to 0}\frac t2\int_{(0,t]} |N_1(x+t-y)|\,\diff y\\
        &= x\int_{(0,x)} |N_1(y)|\,\diff y +\lim_{t\to 0}\frac t2\int_{(0,t]} |N_1(x+z)|\,\diff z.
	\end{align*}
	Since $N_1\in L^1_\loc(\real_+)$, the right-hand side tends to zero as $x\to 0$, thus
	\begin{align*}
	   \lim_{x\to 0}\partial_+(N_1*h)(x)=0;
	\end{align*}
    a similar calculation for the left derivative shows  $\partial_-(N_1*h)=0$. Further, combining Lemma~\ref{lem-h-properties} \ref{lem-h-i} with the assumptions on $\rho_1$ shows that $|h(x)|<Cx^{\alpha-1}$ for $x\in(0,1)$. Thus the assumptions on $N_2$ and Lemma~\ref{convo-asympt}~\ref{convo-asympt-ii} ensure that
	\begin{align}\nonumber
        |\partial(N_2*h)(x)|&
        = |(\partial N_2)*h(x)|\\ &
        \leq C_1C_2\int_{(0,x)}(x-y)^{-\alpha}y^{\alpha-1}\,\diff y
        \leq C_1C_2 B(1-\alpha,\alpha)\label{eq-arg1}
	\end{align}
    for $x\in(0,1)$ and with $C_1,C_2,\alpha>0$ as in Theorem \ref{thm-smoothness}.
	
\medskip\noindent
    \ref{lem-h-iv} This follows from \ref{lem-h-iii} and the fact that $h\in L^1_{\loc}(\real_+)$. Indeed,
    \begin{gather*}
        \lim_{x\to0+}|h*\partial(h*\bar{\bar\nu})(x)|
        \leq \lim_{x\to0+} \left(\sup_{z\in(0,1)}|\partial(h*\bar{\bar\nu})(z)|\right) \int_{(0,x)}|h(y)|\,\diff y
        = 0.
    \end{gather*}
    This completes the proof.
\end{proof}

We are now ready to prove the main result of this section.

\begin{proof}[Proof of Theorem~\ref{thm-smoothness}]
    It suffices to consider the case $q=0$, as otherwise an exponential change of measure allows us to represent $W^{(q)}$ in terms of $W^{(0)}$   as stated in Section \ref{sub-se-23}.

    The function $h$ from Lemma~\ref{lem-h-properties} satisfies the requirements of Theorem \ref{thm-unbdd}. Therefore, we have the representation
    \begin{align*}
	   W: = W^{(0)} =\bone*h*\sum_{n=0}^\infty f^{*n},\quad q\geq 0,
    \end{align*}
    with $f(x) := c''h(x)-\partial(h*\bar{\bar\nu})(x)$ for $x\in (0,\infty)$; moreover,
    \begin{align}\label{eq-W'}
	   W'=h*\sum_{n=0}^\infty f^{*n}.
    \end{align}

    By assumption, $\bar\nu\in\cC^{k_0}(0,\infty)$. Hence, $\bar{\bar\nu}\in\cC^{k_0+1}(0,\infty)$, and, since $h\in L^1_{\loc}(\real_+)$,  Lemma~\ref{lem-Conv-C1} shows that $h*\bar{\bar\nu}\in\cC^{k_0+1}(0,\infty)$. Again with Lemma \ref{lem-Conv-C1} we see that $f^{*n}\in\cC^{k_0}(0,\infty)$ for all $n\in\nat\setminus\{0\}$. Thus, $W\in\cC^{k_0+1}(0,\infty)$ follows if we can show that that the series \eqref{eq-W'} converges in the space $\cC^{k_0}(0,\infty)$, i.e.\ locally uniformly with all derivatives up to order $k_0$.

    Lemma \ref{lem-h-properties} \ref{lem-h-i}, together   with the assumptions on $\rho_1$, shows that there exist constants $C,\alpha>0$ such that $|h(x)|<Cx^{\alpha-1}$ for $x\in(0,1)$. As $\partial(\bar{\bar\nu}*h)$ is bounded on $(0,1)$, see Lemma \ref{lem-h-properties} \ref{lem-h-iii}, there exist constants $\bar{C},\bar{\alpha}>0$ such that $|f(x)|<\bar{C}x^{\bar\alpha-1}$ for $x\in(0,1)$.

Hence, the assumptions of Corollary \ref{corol_deriv} are satisfied, and setting $m:=\lfloor 2/\bar\alpha\rfloor +1$ we have for all $k\leq k_0$ and $n\geq k_0m$
\begin{align*}
	\partial^k f^{*n}=(\partial f^{*m})^{*k}*f^{*(n-km)}.
\end{align*}
By Lemma~\ref{convo-asympt} \ref{convo-asympt-ii} we have $\partial^kf^{*(n+m)}(0)=0$ for all $k\leq k_0$. Thus, we can use Lemma~\ref{lem-Conv-C1} \ref{lem-Conv-C1-i}	to obtain
\begin{align*}
	\partial^k (h*f^{*(n+m)})
    = h*(\partial f^{*m})^{*k}*f^{*(n-(k-1)m)}.
\end{align*}

Returning to the series \eqref{eq-W'}, this means for $N> (k_0+1)m$
\begin{align}
    \partial^k \bigg(h*\sum_{n=0}^N f^{*n}\bigg)
    = \partial^k\bigg(h*\sum_{n=0}^{(k_0+1)m} f^{*n}\bigg) + h*\left(\partial f^{*m}\right)^{*k} * \sum_{\mathclap{n=(k_0+1)m+1}}^{N} f^{*(n-km)}.
\end{align}
The first term is independent of $N$. If we pull out $f^{*m}$ from the sum appearing in the second term, we see that
\begin{align*}
	h*(\partial f^{*m})^{*k}*\sum_{\mathclap{n=(k_0+1)m+1}}^{N} f^{*(n-km)}
    = h*(\partial f^{*m})^{*k}*f^{*m}*\sum_{\mathclap{n=k_0m+1}}^{N-m} f^{*(n-km)}.
\end{align*}
The first three convolution factors can be bounded by
\begin{align*}
	S:=\sup_{x\leq z} |h*(\partial f^{*m})^{*k}*f^{*m}(x)| < \infty \quad\text{for all $z>0$};
\end{align*}
here we use Lemma \ref{convo-asympt} and the continuity of $h$, $(\partial f^{*m})^{*k}$ and $f^{*m}$. Thus, with H\"{o}lder's inequality,
\begin{align*}
    \left|h*(\partial f^{*m})^{*k}*f^{*m}*\sum_{\mathclap{n=k_0m+1}}^{N-m} f^{*(n-km)}(x)\right|
    \leq S \sum_{n=k_0m+1}^{N-m} \left\|f^{*(n-km)}\right\|_{L^1[0, z]}
\end{align*}
for all $x\in[0,z]$. Lemma \ref{lem-ConPowDecay} implies that the sum on the right-hand side converges uniformly for all $x\in [0,z]$, and so the series
\begin{align*}
	\partial^k \bigg(h*\sum_{n=0}^\infty f^{*n}\bigg)
\end{align*}
converges locally uniformly in $(0,\infty)$. This proves $W\in\cC^{k_0+1}(0,\infty)$.

Assume now that $\kappa(\rho_1)=2$. Lemma \ref{lem-h-properties} \ref{lem-h-i} implies that $\kappa(h)=2$, as well. Thus, Lemma \ref{convo-asympt} gives $h^{*2}(0+)=0$. Plugging in the definition of $f$ into the series \eqref{eq-W'} shows that for $x>0$
\begin{align*}
	W'(x)&=\lim_{N\to\infty} h*\sum_{n=0}^N (c''h-\partial(h*\bar{\bar\nu}))^{*n}(x)\\
	&=\lim_{N\to\infty} \sum_{n=0}^N \sum_{\ell=0}^n (-1)^\ell \binom n\ell (c''h)^{*(n-\ell)}*h*(\partial(h*\bar{\bar\nu}))^{*\ell}.
\end{align*}
Each summand contains a factor of the form $h*(\partial(h*\bar{\bar\nu}))^{*\ell}$ for some $\ell\in\nat$, and we will now see how smooth it is. The case $\ell=0$ is trivial, as $h\in\cC^\infty(0,\infty)$. For $\ell=1$ we have, due to Lemma \ref{lem-Conv-C1},
\begin{align}\label{eq-smooth-induction}
	h*\partial(h*\bar{\bar\nu})=\partial(h^{*2}*\bar{\bar\nu})-h=\partial(h^{*2})*\bar{\bar\nu}-h \in\cC^{k_0+1}(0,\infty).
\end{align}
Finally, if $\ell\geq2$, we use
\begin{align*}
	h*(\partial(h*\bar{\bar\nu}))^{*\ell}=h*\partial(h*\bar{\bar\nu})*(\partial(h*\bar{\bar\nu}))^{*(\ell-1)}.
\end{align*}
With \eqref{eq-smooth-induction}, Lemma \ref{lem-h-properties} \ref{lem-h-iv}, and Lemma \ref{lem-Conv-C1} we obtain
\begin{align*}
    \partial\left(h*(\partial(h*\bar{\bar\nu}))^{*\ell}\right)
    = \partial\left(h*\partial(h*\bar{\bar\nu})\right)*(\partial(h*\bar{\bar\nu}))^{*(\ell-1)}
    \in\cC^{k_0}(0,\infty).
\end{align*}
This shows that each summand of
\begin{align*}
    \sum_{n=0}^N \sum_{\ell=0}^n(-1)^\ell  \binom n\ell (c''h)^{*(n-\ell)}*h*(\partial(h*\bar{\bar\nu}))^{*\ell}
\end{align*}
is $(k_0+1)$ times continuously differentiable on $(0,\infty)$. As in the first case we see that this smoothness in inherited by the series \eqref{eq-W'}. Consequently, $W\in\cC^{k_0+2}(0,\infty)$, and the proof is complete.
\end{proof}

\subsection{Discussion}\label{sub-se-42}

We begin the section with an example that illustrates Theorem \ref{thm-smoothness}.

\begin{example}\label{example-45}
    Denote by $(L^{(i)}_t)_{t\geq0}, i=1,2$ two independent spectrally negative L\'{e}vy processes with characteristic triplets $(c^{(i)},0,\widetilde\nu^{(i)}), i=1,2$. The characteristic triplet of the sum $L_t:=L^{(1)}_t+L^{(2)}_t$ is then $(c^{(1)}+c^{(2)},\widetilde\nu^{(1)}+\widetilde\nu^{(2)})$. For $i=1,2$, denote by $N_i$ the integrated tail function of the reflection $\nu^{(i)}$ of the L\'{e}vy measure of $L^{(i)}_t$.

   Let $N_1$ be completely monotone, and assume that there exist constants $C>0,   \alpha\in(0,1)  $ such that $N_1(x)\geq Cx^{-\alpha}$ for $x\in(0,1)$. By Corollary \ref{cor-cm-density}, $N_1$ admits a completely monotone resolvent $\rho_1\in\cC^\infty(0,\infty)$ for which
	\begin{align*}
		\rho_1(x) \frac{C}{1-\alpha}= x^{  \alpha-1  }\rho_1(x) \int_{(0,x)} Cy^{-\alpha}\,\diff y \leq \int_{(0,x)} \rho_1(x-y) N_1(y)\,\diff y =1
	\end{align*}
    for $x\in(0,1)$. Thus, $\rho_1(x)<\frac{1-\alpha}Cx^{\alpha-1}$ on $(0,1)$, and condition \ref{thm-smooth-i} in Theorem \ref{thm-smoothness} is met.

    This shows that the smoothness of the scale functions $W^{(q)}$ of the sum $L_t$ depends on the smoothness of $N_2$ only, i.e. on the L\'{e}vy measure of $L^{(2)}_t$. For example, $L^{(2)}_t$ could be such that $-L^{(2)}_t$ is a subordinator, and with L\'{e}vy measure $\nu^{(2)}$ such that $\bar\nu^{(2)}<C_2x^{-\alpha}$ on $(0,1)$ for some constant $C_2>0$. Condition \ref{thm-smooth-ii} of Theorem \ref{thm-smoothness} is then easily verified.

    We conclude that in this case, if $\bar\nu^{(2)}\in \cC^{k_0}(0,\infty) $ for $k_0\in\nat$, then $W^{(q)}\in\cC^{k_0+1}(0,\infty)$ and even $W^{(q)}\in\cC^{k_0+2}(0,\infty)$, if additionally $\alpha\geq \frac12$.
\end{example}

The following simple corollary of Theorem \ref{thm-smoothness} shows that the implication
\begin{align*}
	\bar\nu\in\cC^n(0,\infty) \implies W^{(q)}\in\cC^{n+\kappa(\rho_1)}(0,\infty)
\end{align*}
holds, if the $0$-scale function of the compensated process is at least two times continuously differentiable and of bounded growth.

\begin{corollary}\label{cor-comp}
    Let $(L_t)_{t\geq 0}$ be a spectrally negative L\'evy process with triplet $(c,0,\widetilde\nu)$ such that $\int_{(0,1)} y\,\nu(\diff y) = \infty$ and $\int_{[1,\infty)} y\,\nu(\diff y) < \infty$. Let $\bar\nu\in\cC^{k_0}(0,\infty)$ for some $k_0\in\nat$. Further, denote by $\widetilde W$ the $0$-scale function of the compensated process $\widetilde L_t:=L_t-\EE[L_t]$. If $\widetilde W\in\cC^2(0,\infty)$ and $|\widetilde W'(x)|<Cx^{\alpha-1}$ for some constants $C,\alpha>0$, then $W^{(q)}\in\cC^{k_0+\kappa(\widetilde W')}(0,\infty)$.
\end{corollary}

\begin{proof}
    By Lemma \ref{lem-compensated} $\widetilde{W}'$ is the resolvent of $\bar{\bar\nu}$. Thus, the assertion follows immediately from Theorem \ref{thm-smoothness} setting $N_2\equiv 0$.
\end{proof}

Both Theorem \ref{thm-smoothness}, and the two earlier mentioned results from \cite{chan11} on the smoothness of scale functions of spectrally negative L\'{e}vy processes with either a Gaussian component or paths of bounded variation, require a growth condition of the type
\begin{align*}
    |f(x)|<Cx^{\alpha-1},\quad x\in(0,1),
\end{align*}
for some constants $C,\alpha>0$ on one of the occurring terms. Getting rid of this requirement seems to be the biggest challenge in a possible proof of Doney's conjecture.

In this regard, the next corollary gives a hint how the growth condition on $\rho_1$ in Theorem \ref{thm-smoothness} can be avoided by imposing stronger conditions on the component $N_2$.

\begin{corollary}
    Let $(L_t)_{t\geq 0}$ be a spectrally negative L\'evy process with triplet $(c,0,\widetilde\nu)$ such that $\int_{(0,1)} y\,\nu(\diff y) = \infty$ and $\int_{[1,\infty)} y\,\nu(\diff y) < \infty$. Let $\bar\nu\in\cC^{k_0}(0,\infty)$ for some $k_0\in\nat$, and assume that there exists a decomposition $\bar{\bar\nu}=N_1+N_2$ and a constant $C_2>0$  such that
    \begin{enumerate}
        \item[$(\textup{i})$] $N_1\in L^1_{\loc}(\real_+)$ admits a resolvent $\rho_1\in\cC^1(0,\infty)$ and
        \item[$(\textup{ii})$] $N_2\in\cC^{k_0+1}(0,\infty)$ with $N_2(0+) =-c'' $ and $|\partial N_2(x)|<C_2$ for $ x\in(0,1)$.
    \end{enumerate}
    Then $W^{(q)}\in\cC^{k_0+\kappa(\rho_1)}(0,\infty)$ for all $q\geq0$.
\end{corollary}

\begin{proof}
    Without loss of generality we assume that $q=0$ and consider $W:=W^{(0)}$. By Remark \ref{rem-cond-relax} we can assume $N_2(0+)=0$ which then yields that the function $f$ in the series expansion \eqref{eq-W'} of $W'$ has to be replaced by
    \begin{equation} \label{eqfneu}
 	  f(x)= -\partial(h*(\bar{\bar\nu}+c''))(x).
    \end{equation}
    Thus, the term $c''h(x)$ vanishes due to the assumptions on $N_2$.

    We can now follow the lines of the proof  of Theorem \ref{thm-smoothness} and observe that the growth condition on $\rho_1$ appears at exactly two places:

    Firstly, in Lemma \ref{lem-h-properties} \ref{lem-h-iii} for the argument \eqref{eq-arg1}, which can now be replaced by
        \begin{align*}
            |\partial(N_2*h)(x)|
            = |\partial N_2*h(x)|\leq C
    \end{align*}
    due to Lemma \ref{convo-asympt}. Thus, Lemma \ref{lem-h-properties} \ref{lem-h-iii} still holds true, and therefore, $\partial(h*(\bar{\bar\nu}+c''))(x)$ is bounded on $(0,1)$. Note that $\bar{\bar\nu}+c''$ plays the role of $\bar{\bar\nu}$, see Remark \ref{rem-cond-relax}.

    Secondly, it is needed to apply Corollary \ref{corol_deriv} to the function $f$  in the series expansion of $W$. But the requirements of Corollary \ref{corol_deriv} are already met as $f$ stays bounded in $(0,1)$ due to the argument above.
\end{proof}

\noindent
Finally, we mention here that the close connection between scale functions and Volterra resolvents  of the first kind as dicussed in Section \ref{S4c} allows for a smoothness result on the latter as well.

\begin{corollary}\label{cor-resolvent}
	Let $f\in L^1_{\loc}(\real_+)$ be a positive, decreasing and convex function with $f(0+)=\infty$ such that $\partial^2 f(x)\,\diff x$ is a L\'evy measure. Denote by $\rho$ the resolvent of $f$.
	\begin{enumerate}
    \item\label{cor-resolvent-i}
        If $\rho\in\cC^1(0,\infty)$ and if there are constants $C,\alpha>0$ such that $|\rho(x)|<Cx^{\alpha-1}$ then
		\begin{align*}
			f\in\cC^n(0,\infty) \Rightarrow \rho \in \cC^{n-2+\kappa(\rho)}(0,\infty)
		\end{align*}		
		for all $n\geq2$.
	
    \item\label{cor-resolvent-ii}  If Doney's conjecture \ref{doney-iii} \textup{(}see p.~\pageref{doney-iii}\textup{)} is true, then
		\begin{align*}
			f\in\cC^n(0,\infty) \iff \rho \in\cC^n(0,\infty)
		\end{align*}
		for all $n\in\nat$.
	\end{enumerate}
\end{corollary}
\begin{proof}
    The assumptions on $f$ ensure that $\partial^2 f(x)\,\diff x$ is the reflection of the L\'evy measure of a spectrally negative L\'evy process with paths of unbounded variation while the requirements for $\rho$ in \ref{cor-resolvent-i} imply that the spectrally negative L\'{e}vy process with Laplace exponent
	\begin{align*}
	   \psi(\beta)=\beta^2\cL[f](\beta)
	\end{align*}
    fulfils the assumptions of Theorem \ref{thm-smoothness}. Assertion \ref{cor-resolvent-i} now follows directly as $\rho=\partial W^{(0)}$ while \ref{cor-resolvent-ii} follows directly from Doney's conjecture without further assumptions on $\rho$.
\end{proof}

\section{Auxiliary results} \label{S3}


In this section we collect a few technical results and their proofs. We believe that this material  is  well-known, but we could not pinpoint suitable references, and so we decided to include the arguments for the reader's convenience.
\begin{lemma}\label{lem-Conv-C1}
Let $k\in\nat$.
\begin{enumerate}
\item\label{lem-Conv-C1-i}
    If $f\in \cC^{k+1}(0,\infty)$ be  such that $\partial^j f(0+)\in\real$ exists for all $j=0,1,\dots,k$,  and $g\in L^1_{\loc}(\real_+)\cap\cC^k(0,\infty)$, then $f*g\in \cC^{k+1}(0,\infty)$ and
	\begin{align}\label{eq-derivativeConvolution}
	   \partial^{\ell}(f*g)(x)
        = (\partial^\ell f)*g(x)+\sum_{j=0}^{\ell-1} \partial^{j} f(0+) \partial^{\ell-1-j} g(x),
        \; \ell\in\{1,\ldots, k+1\}.
	\end{align}
\item\label{lem-Conv-C1-ii}
    If $f,g\in L^1_{\loc}(\real_+)\cap\cC^k(0,\infty)$ then $f*g\in\cC^k(0,\infty)$.
\end{enumerate}	
\end{lemma}
\begin{proof}
\ref{lem-Conv-C1-i}
    Formula \eqref{eq-derivativeConvolution} follows with a straightforward calculation using integration by parts; since the terms on the right-hand side  are finite and continuous, we see that $f*g\in\cC^{k+1}(0,\infty)$.

\medskip\noindent
\ref{lem-Conv-C1-ii}
    Note that $f$ and $\partial^j f$ may have singularities at $x=0$, so that the formula \eqref{eq-derivativeConvolution} need not hold. We can, however argue as follows: let $x\geq 0$ and fix $\delta<x/2$. We have
	\begin{align*}
		f*g(x)
        =\int_0^\delta f(x-y)g(y)\,\diff y + \int_\delta^{x-\delta}f(x-y)g(y)\,\diff y + \int_{x-\delta}^xf(x-y)g(y)\,\diff y.
	\end{align*}
	The first and the third term are $k$ times continuously differentiable, while the second term can be written as
	\begin{align*}
		\int_\delta^{x-\delta}f(x-y)g(y)\,\diff y
        = \int_0^{x-2\delta} f(x-2\delta-y+\delta)g(y+\delta)\,\diff y
        = f_\delta*g_\delta(x-2\delta)
	\end{align*}
    for the shifted functions $f_\delta(\cdot):=f(\cdot+\delta)$ and $g_\delta(\cdot):=g(\cdot+\delta)$. By assumption, $f_\delta, g_\delta\in\cC^k(\real_+)$ are such that we can apply Part~\ref{lem-Conv-C1-i}, showing that $f_\delta*g_\delta\in\cC^k(0,\infty)$, and the claim follows.
\end{proof}

\begin{lemma}\label{convo-asympt}
    Let $f_1,f_2\in L^1_{\loc}(\real_+)$ be two functions such that $|f_1(x)|<C_1 x^{\alpha_1-1}$ for $x$ in a neighbourhood of zero and some constants $C_1,\alpha_1>0$.
	\begin{enumerate}
	\item\label{convo-asympt-i}
        If, additionally, $|f_2(x)|<C_2 x^{\alpha_2-1}$ for $x$ in a neighbourhood of zero and some constants $C_2,\alpha_2>0$, then $f_1*f_2$ exists and satisfies $|f_1*f_2(x)|<C x^{\alpha_1 +\alpha_2-1}$ for $x$ in a neighbourhood of zero and some constant $C>0$. 	
    \item\label{convo-asympt-ii}
        If $\alpha_1\geq1$, then $f_1*f_2$ exists and there exists a constant $C>0$ such that $|f_1*f_2(x)|<C  x^{\alpha_1-1}$ for all $x$ in a neighbourhood of zero.
	\end{enumerate}
\end{lemma}
\begin{proof}
	\ref{convo-asympt-i} By assumption, we find for all $x$ in a neighbourhood of zero
	\begin{align*}
		|f_1*f_2(x)|
        &\leq \int_0^x |f_1(x-y)||f_2(y)|\,\diff y
        \leq C_1C_2\int_0^x (x-y)^{\alpha_1-1}y^{\alpha_2-1}\,\diff y.
	\end{align*}
	As $\alpha_1,\alpha_2>0$ the assertion follows after an obvious change of variables.

\medskip\noindent
	\ref{convo-asympt-ii} The second assertion follows from H\"{o}lder's inequality as
	\begin{align*}
		|f_1*f_2(x)|
        \leq \int_0^x |f_2(x-y)|y^{\alpha_1-1}\,\diff y
        \leq \|f_2\|_{L^1([0,x])} \sup_{y\in[0,x]} |y^{\alpha_1-1}|
        = \|f_2\|_{L^1([0,x])} x^{\alpha_1-1},
	\end{align*}
    for $x\geq 0$ small enough.
\end{proof}

\begin{corollary}\label{corol_deriv}
    Let $f\in\cC^1(0,\infty)$ and assume that $|f(x)|<C x^{\alpha-1}$ for $x\in(0,1)$ and some constants $C,\alpha>0$. For all $k\in\nat$ and $n\geq k(\lfloor 2/\alpha\rfloor+1)$ one has $f^{*n}\in\cC^k(0,\infty)$. In particular,
	\begin{gather*}
        \partial^k  f^{*n} = (\partial  f^{*m} )^{*k}*f^{*(n-km)}
        \quad\text{with\ \ } m=\lfloor 2/\alpha\rfloor+1.
    \end{gather*}
\end{corollary}
\begin{proof}
     The case $k=0$ is trivial. Let $k=1$.  Lemma~\ref{convo-asympt} \ref{convo-asympt-i} shows that there is some constant $C'>0$ such that $|f^{*m}(x)|<C'x$ for all sufficiently small $x<1$.   Therefore, by Lemma \ref{lem-Conv-C1} \ref{lem-Conv-C1-i}, $\partial f^{*n}=(\partial f^{*m})*f^{*(n-m)}$ for $n \geq m$. If $k>1$, the assertion follows by iteration.
\end{proof}

The following lemma is essential for our main results as it ensures uniform convergence on compact subsets of $\real_+$ of the series expansions of the $q$-scale functions.
\begin{lemma}\label{lem-ConPowDecay}
Let $f\in L^1_{\loc}(\real_+)$, $z>0$ and $\epsilon>0$. Then the limit
		\begin{align}\label{eq-ConPowDecay}
			\lim_{n\to\infty} \frac{\|f^{*n}\|_{L^1([0,x])}}{\epsilon^n} =0.
		\end{align}	
		exists uniformly for all $x\in[0,z]$.
\end{lemma}
\begin{proof}
	Fix $\epsilon,x>0$.   Since for $y\in[0,1]$ the convolution satisfies $\int_0^y f(y-z)f(z)\diff z = f*f(y)=(f\bone_{[0,x]})*(f\bone_{[0,x]})(y)$, we may assume  that $f(y)=0$ for $y>x$ and for fixed $x\in\real_+$. This ensures, in particular, that the Laplace transform of $f$ exists.   For all $\beta\geq 0$ we have
	\begin{align*}
		\|f^{*n}\|_{L^1([0,x])}\leq \ee^{\beta x} \cL[|f^{*n}|](\beta) \leq \ee^{\beta x} \left(\cL[|f|](\beta)\right)^n,
	\end{align*}
	and the assertion follows since there exists some $\beta_0\geq 0$ such that $\cL[|f|](\beta_0)<\frac 12\epsilon$.
\end{proof}


\subsubsection*{Acknowledgements}
\noindent
We would like to thank the reviewers for their helpful and constructive comments when preparing the revision of this paper.

\subsubsection*{Funding information} 
\noindent
R.L.\ Schilling gratefully acknowledges financial support through the DFG-NCN Beet\-hoven Classic 3 project SCHI419/11-1 \& NCN 2018/31/G/ST1/02252.

\subsubsection*{Competing interests} 
\noindent There were no competing interests to declare which arose during the preparation or publication process of this article.

\end{document}